\documentclass[english,aap,preprint]{imsart}
\usepackage[latin9]{inputenc}
\usepackage{float}
\usepackage{amsthm}
\usepackage{amsmath}
\usepackage{amssymb}
\usepackage{graphicx}
\usepackage{esint}
\usepackage[authoryear]{natbib}

\makeatletter

\floatstyle{ruled}
\newfloat{algorithm}{tbp}{loa}
\providecommand{\algorithmname}{Algorithm}
\floatname{algorithm}{\protect\algorithmname}

\numberwithin{equation}{section}
\numberwithin{figure}{section}
  \theoremstyle{plain}
  \newtheorem{fact}{\protect\factname}
  \theoremstyle{plain}
  \newtheorem{prop}{\protect\propositionname}
  \theoremstyle{remark}
  \newtheorem{rem}{\protect\remarkname}
\theoremstyle{plain}
\newtheorem{thm}{\protect\theoremname}
  \theoremstyle{plain}
  \newtheorem{lem}{\protect\lemmaname}
  \theoremstyle{plain}
  \newtheorem{assumption}{\protect\assumptionname}

\RequirePackage[OT1]{fontenc}
\RequirePackage{amsthm,amsmath,natbib}
\RequirePackage[colorlinks,citecolor=blue,urlcolor=blue]{hyperref}

\usepackage{enumerate}
\usepackage{tabularx}
\usepackage{booktabs}
\usepackage{multirow}

\makeatother

\usepackage{babel}
  \providecommand{\assumptionname}{Assumption}
  \providecommand{\factname}{Fact}
  \providecommand{\lemmaname}{Lemma}
  \providecommand{\propositionname}{Proposition}
  \providecommand{\remarkname}{Remark}
\providecommand{\theoremname}{Theorem}

\begin{document}
\def\abstractname{}
\begin{frontmatter}
\title{Perfect sampling for nonhomogeneous Markov chains and hidden Markov models}
\runtitle{Perfect sampling for nonhomogeneous Markov chains}
\begin{aug}
  \author{\fnms{Nick} \snm{Whiteley} \ead[label=e1]{nick.whiteley@bristol.ac.uk}}
  \address{School of Mathematics\\ University of Bristol\\ University Walk\\ Bristol BS8 1TW\\ United Kingdom.\\ \printead{e1}}
  \and
  \author{\fnms{Anthony} \snm{Lee}\ead[label=e2]{anthony.lee@warwick.ac.uk}}
  \address{Department of Statistics\\ University of Warwick\\ Coventry CV4 7AL\\ United Kingdom.\\ \printead{e2}}
  \affiliation{University of Bristol and University of Warwick}
  \runauthor{N. Whiteley and A. Lee}

\end{aug}



\thankstext{t1}{The first author is supported by EPSRC Grant EP/K023330/1}
\begin{abstract}
We obtain a perfect sampling characterization of weak ergodicity for
backward products of finite stochastic matrices, and equivalently,
simultaneous tail triviality of the corresponding nonhomogeneous Markov
chains. Applying these ideas to hidden Markov models, we show how
to sample exactly from the finite-dimensional conditional distributions
of the signal process given infinitely many observations, using an
algorithm which requires only an almost surely finite number of observations
to actually be accessed. A notion of ``successful'' coupling is
introduced and its occurrence is characterized in terms of conditional
ergodicity properties of the hidden Markov model and related to the
stability of nonlinear filters.
\end{abstract}
\end{frontmatter}

\section{Introduction\label{sec:Introduction}}

With the introduction of their famous Coupling From the Past (CFTP)
algorithm, \citet{propp1996exact} showed how to use a form of backward
coupling to simulate exact samples from the invariant distribution
of an ergodic Markov chain in a.s. finite time. \citet{foss1998perfect},
in part appealing to a construction of \citet{murdoch1998exact},
showed that existence of an a.s. finite backward coupling time characterizes
uniform geometric ergodicity of the Markov chain in question. The
present papers extends these ideas in the context of nonhomogeneous
Markov chains, a setting which to date has received little attention,
perhaps due to a lack of appropriate formulation or applications.

Our contribution is to present such a formulation and apply the insight
which we develop about nonhomogeneous chains to hidden Markov models
(HMM's), for which we obtain a perfect sampling characterization of
\emph{conditional ergodicity} phenomena, i.e. ergodic properties of
the signal process in the HMM under its conditional law given the
observations, along the lines of those addressed by \citet{van2009stability}.
Even for HMM's with finite state space, conditional ergodicity and
the connection to perfect sampling can be subtle, due to the delicate
interplay between the observations and signal in the HMM, and the
fact that the ergodic theory of nonhomogeneous Markov chains, which
governs the behavior of the signal when conditioned on observations,
is considerably more complicated than that of homogeneous chains.

\subsection{Nonhomogeneous Markov chains and backward products\label{sub:Nonhomogeneous-Markov-chains}}

One of the key notions underlying CFTP is that if an ergodic Markov
chain were initialized infinitely far in the past and run forward
in time, its state at the present would be distributed exactly according
to the invariant distribution of the chain. In order to give an overview
of our main results, we need to identify a suitable and somewhat elementary
generalization of this notion.

Let $\mathbb{T}:=\mathbb{Z}^{-}\cup\{0\}$ be the set of nonpositive
integers and let $M=(M_{n})_{n\in\mathbb{T}}$ be a sequence of Markov
kernels on a finite set $E=\{1,\ldots,s\}$, so for each $x\in E$,
$M_{n}(x,\cdot)$ is probability distribution on $E$. One may construct
a non-homogeneous Markov chain $(X_{n})_{n\in\mathbb{T}}$ with paths
in $E^{\mathbb{T}}$ and transitions given by $M$ in the sense that
$X_{n}|X_{n-1}\sim M_{n}(X_{n-1},\cdot)$, as soon as there exists
a sequence $\pi=(\pi_{n})_{n\in\mathbb{T}}$ of \emph{absolute probabilities}:\emph{
}a family of probability distributions with the property that for
all $n\in\mathbb{T}$ and $x\in E$,
\[
\sum_{z\in E}\pi_{n-1}(z)M_{n}(z,x)=\pi_{n}(x).
\]
Indeed one can then readily define a consistent family of finite dimensional
distributions $(\mathbf{P}_{\pi}^{(n)})_{n\in\mathbb{T}}$,
\begin{equation}
\mathbf{P}_{\pi}^{(n)}(X_{n}=x_{n},\ldots,X_{0}=x_{0}):=\pi_{n}(x_{n})\prod_{k=n+1}^{0}M_{k}(x_{k-1},x_{k}),\label{eq:P^n_introc}
\end{equation}
giving rise via the usual Kolmogorov extension argument to a probability
measure $\mathbf{P}_{\pi}$ over paths in $E^{\mathbb{T}}$; one can
think of $(X_{n})_{n\in\mathbb{T}}$ running forward in time from
the distant past towards zero.

Now for $k\in\mathbb{T}$ define recursively
\begin{equation}
M_{k,k}:=Id,\quad\quad M_{n-1,k}(x,x^{\prime}):=\sum_{z\in E}M_{n}(x,z)M_{n,k}(z,x^{\prime}),\quad n\leq k.\label{eq:recurse M}
\end{equation}
With $k\in\mathbb{T}$ fixed, $(M_{n,k})_{n\leq k}$ are called\emph{
backward products}, since they can be written in terms of matrix multiplications
to the left: $M_{n-1,k}=M_{n}M_{n,k}$.

Questions of existence and uniqueness of $\pi$, and thus of $\mathbf{P}_{\pi}$,
are answered with the following long-established facts, which hold
for any sequence of Markov kernels $M=(M_{n})_{n\in\mathbb{T}}$ on
a finite state space $E$, see \citep[Section 4.6]{seneta2006non}
and references therein for an accessible introduction.
\begin{fact}
\label{fact:there-always-exists}there always exists at least one
sequence of absolute probabilities
\end{fact}

\begin{fact}
\label{fact:weak_ergodicity}there exists a unique sequence of absolute
probabilities if and only if the backward products of $M$ are \emph{weakly
ergodic}, meaning
\begin{equation}
\lim_{n\rightarrow-\infty}M_{n,k}(x,z)-M_{n,k}(x^{\prime},z)=0,\quad\quad\forall\; k\in\mathbb{T},\;(z,x,x^{\prime})\in E^{3},\label{eq:weak_ergo_M}
\end{equation}
in which case,
\begin{equation}
\lim_{n\rightarrow-\infty}M_{n,k}(x,z)-\pi_{k}(z)=0,\quad\quad\forall\; k\in\mathbb{T},\;(z,x)\in E^{2},\label{eq:conv_to_pi_intro}
\end{equation}
where $\pi=(\pi_{n})_{n\in\mathbb{T}}$ is the unique sequence of
absolute probabilities for $M$.
\end{fact}

\subsection{Perfect sampling and characterizations of weak ergodicity\label{sub:Perfect-sampling-and}}

Our basic algorithmic goal, when weak ergodicity holds, is to obtain
exact draws from each $\pi_{n}$. This can be achieved with a very
modest generalization of Propp and Wilson's method. The only existing
works on perfect simulation for nonhomogeneous chains which we know
of are \citep{glynn2001two} and \citep{stenflo2007perfect}, which
respectively provide perfect sampling methods for Markov chains conditioned
to avoid certain states, and products of transition matrices subject
to a particular uniform regularity assumption, which we discuss in
more detail later. Our first goal is to develop more general insight
into how the feasibility of CFTP for nonhomogeneous chains is related
to various ergodic properties of $M$ and $\mathbf{P}_{\pi}$.

Inspired by \citet{foss1998perfect}'s characterization of uniform
geometric ergodicity for a homogeneous chain in terms of the existence
of a successful (meaning a.s. finite) backward coupling time, we assert
that ``success'' in the nonhomogeneous case is for not just one,
but all of a particular countably infinite family of coupling times
to be a.s. finite. Our first main result, Theorem~\ref{thm:perfect sampling-1},
shows that success so-defined of our coupling is equivalent to weak
ergodicity, as in (\ref{eq:weak_ergo_M}), which is weaker than the
assumption of \citet{stenflo2007perfect}, and if successful our coupling
delivers a sample from each member of the then unique sequence of
absolute probabilities $(\pi_{n})_{n\in\mathbb{T}}$ in a.s. finite
time.

We extend this ergodic characterization in Theorem~\ref{thm:tail},
by showing that unicity of a sequence of absolute probabilities, hence
weak ergodicity, hence success of our coupling, is also equivalent
to the \emph{simultaneous tail triviality} condition:
\begin{align}
 & \mathbf{P}_{\pi}(A)=\mathbf{P}_{\pi}(A)^{2}=\mathbf{P}_{\tilde{\pi}}(A),\quad\quad\label{eq:tail_triv_front}\\
 & \quad\quad\quad\quad\quad\quad\quad\quad\quad\forall(\pi,\tilde{\pi},A)\in\Pi_{M}\times\Pi_{M}\times\bigcap_{n\in\mathbb{T}}\sigma(X_{k};k\leq n),\nonumber
\end{align}
where $\Pi_{M}$ is the set of all sequences of absolute probabilities
for $M$.

\subsection{Hidden Markov models and conditional ergodicity}

Our motivation for considering nonhomogeneous chains and the condition
(\ref{eq:tail_triv_front}) stems from hidden Markov models, which
are widely applied across econometrics, genomics, signal processing
and many other disciplines as they provide flexible and interpretable
means to model dependence between observed data in terms of an unobserved
Markov chain. We take a slightly non-standard perspective in that
an HMM is for us a process $(X_{n},Y_{n})_{n\in\mathbb{T}}$ on a
nonpositive time horizon, where the \emph{signal} $X=(X_{n})_{n\in\mathbb{T}}$
is a possibly nonhomogeneous Markov chain with paths in $E^{\mathbb{T}}$,
and the \emph{observations} $Y=(Y_{n})_{n\in\mathbb{T}}$ are conditionally
independent given $X$, with each $Y_{n}$ valued in a Polish space
$F$ and having a conditional distribution given $X$ which depends
only on $X_{n}$.

With the law of $(X_{n},Y_{n})_{n\in\mathbb{T}}$ then written as
$\mathbf{P}$, a standard task in applications of HMM's is to calculate
conditional distributions of the form $\mathbf{P}(X_{n}\in\cdot\,|\sigma(Y_{k};k\in I))$
where $I$ is some finite subset of $\mathbb{T}$. These distributions
are immediately useful for inference about the signal process and
for making predictions about future observations given those recorded
up to the present. For example, if one extends the HMM onto a positive
time horizon by introducing $(X_{1},Y_{1})$ such that $\mathbf{P}(X_{1}=x,Y_{1}\in\cdot\,|\mathcal{F}^{X}\vee\mathcal{F}^{Y})=M(X_{0},x)G(x,\cdot)$,
where $\mathcal{F}^{X}=\sigma(X_{k};k\in\mathbb{T})$, $\mathcal{F}^{Y}=\sigma(Y_{k};k\in\mathbb{T})$,
and $M$ and $G$ are probability kernels respectively from $E$ to
itself and from $E$ to $F$, then for any $I\subset\mathbb{T}$,
\begin{eqnarray}
 &  & \mathbf{P}(Y_{1}\in\cdot\,|\sigma(Y_{k};k\in I))\label{eq:predict}\\
 &  & \quad\quad\quad\quad\quad\quad=\sum_{x,x^{\prime}}\mathbf{P}(X_{0}=x|\sigma(Y_{k};k\in I))M(x,x^{\prime})G(x^{\prime},\cdot).\nonumber
\end{eqnarray}
When calculating such distributions in order to make predictions,
it is desirable to impart as much information from the past as possible.
For example, the mean-square optimal $\mathcal{F}^{Y}$-measurable
predictor of $Y_{1}$ is, of course, the conditional expectation $\mathbf{E}[Y_{1}|\mathcal{F}^{Y}]$.
However exact calculation of $\mathbf{P}(X_{0}\in\cdot\,|\mathcal{F}^{Y})$,
or indeed $\mathbf{P}(X_{n}\in\cdot\,|\mathcal{F}^{Y})$ for any $n\in\mathbb{T}$,
requires an infinite number of observations to be recorded and in
general cannot be accomplished in finite time.

Nevertheless, under certain conditions, our perfect sampling method
makes it possible to obtain an exact draw from $\mathbf{P}(X_{n}\in\cdot\,|\mathcal{F}^{Y})$
using an algorithm which runs for $|T_{n}|$ time-steps and uses only
$(Y_{k};k=T_{n},T_{n}+1,\ldots,0)$, where $T_{n}$ is a $\mathbb{T}$-valued
random time, the details of which we make precise later. If with $n=0$,
the resulting sample from $\mathbf{P}(X_{0}\in\cdot\,|\mathcal{F}^{Y})$
is denoted $X_{0}^{\star}$ and one also samples $Y_{1}^{\star}\sim\sum_{x}M(X_{0}^{\star},x)G(x,\cdot)$,
then by (\ref{eq:predict}) with $I=\mathbb{T}$, $Y_{1}^{\star}$
is distributed exactly according to the ``ideal'' predictive conditional
distribution $\mathbf{P}(Y_{1}\in\cdot\,|\mathcal{F}^{Y})$.

The connection to Sections~\ref{sub:Nonhomogeneous-Markov-chains}--\ref{sub:Perfect-sampling-and}
arises from the facts that:
\begin{eqnarray*}
 &  & \mathbf{P}(X_{n}\in\cdot\,|\mathcal{F}^{Y}\vee\sigma(X_{k};k<n)),\\
 &  & =\mathbf{P}(X_{n}\in\cdot\,|\mathcal{F}^{Y}\vee\sigma(X_{n-1})),\\
 &  & =\mathbf{P}(X_{n}\in\cdot\,|\sigma(Y_{k};k\geq n)\vee\sigma(X_{n-1})),\quad\mathbf{P}-\text{a.s.},
\end{eqnarray*}
i.e., conditional on the observations, the signal process $X$ is
a nonhomogeneous Markov chain, and its conditional transition probabilities
at time $n$ depend on $Y$ only through $(Y_{k})_{k\geq n}$. Moreover,
for any $y\in F^{\mathbb{T}}$ we can calculate and sample from each
of a family of Markov kernels $M^{y}=(M_{n}^{y})_{n\in\mathbb{T}}$,
with $M_{n}^{Y}(X_{n-1},\cdot)$ a version of $\mathbf{P}(X_{n}\in\cdot\,|\sigma(Y_{k};k\geq n)\vee\sigma(X_{n-1}))$,
for which conditional probabilities of the form $\mathbf{P}(X_{n}\in\cdot\,|\mathcal{F}^{Y})$
define a sequence of absolute probabilities, and $T_{n}$ as mentioned
above is one of a collection of coupling times arising from CFTP applied
to $M^{y}$.

Building from the considerations of Section~\ref{sub:Perfect-sampling-and},
our attention then turns to the question of how success of our HMM
sampling scheme, meaning that every $T_{n}$ is conditionally a.s.-finite
given $Y$, is related to the ergodic properties of the HMM. Compared
to the setup of Sections~\ref{sub:Nonhomogeneous-Markov-chains}--\ref{sub:Perfect-sampling-and},
we have to handle the additional complication here that each $\mathbf{P}(X_{n}\in\cdot\,|\mathcal{F}^{Y})$,
$M_{n}^{Y}$ and success itself depend $Y$. Our main result in this
regard, Theorem~\ref{thm:cond_ergo_perf_sampling}, establishes that
success for $\mathbf{P}$-almost all $Y$ is equivalent to the following
condition, which can be considered the HMM-counterpart of (\ref{eq:tail_triv_front}):
there exists an event $H\in\mathcal{F}^{X}\otimes\mathcal{F}^{Y}$
with $\mathbf{P}(H)=1$ such that for all $\omega=(x,y)\in H$,
\begin{equation}
\mathbf{P}^{\mathcal{F}^{Y}}(\omega,A)=\mathbf{P}^{\mathcal{F}^{Y}}(\omega,A)^{2},\quad\forall A\in\bigcap_{n\in\mathbb{T}}\sigma(X_{k};k\leq n)\label{eq:cond_triviality_front}
\end{equation}
and
\begin{eqnarray}
 &  & \mathbf{P}_{\pi^{Y(\omega)}}(A)=\mathbf{P}^{\mathcal{F}^{Y}}(\omega,A),\label{eq:equal_tail_front}\\
 &  & \quad\quad\quad\quad\quad\quad\quad\quad\quad\quad\forall(\pi^{Y(\omega)},A)\in\Pi_{M^{Y(\omega)}}\times\bigcap_{n\in\mathbb{T}}\sigma(X_{k};k\leq n).\nonumber
\end{eqnarray}
Here, with $\Omega:=E^{\mathbb{T}}\times F^{\mathbb{T}}$, $\mathbf{P}^{\mathcal{F}^{Y}}:\Omega\times\mathcal{F}^{X}\rightarrow[0,1]$
is a probability kernel such that for all $A\in\mathcal{F}^{X}$,
$\mathbf{P}^{\mathcal{F}^{Y}}(\omega,A)=\mathbf{P}(A|\mathcal{F}^{Y})(\omega)$,
$\mathbf{P}-\text{a.s.}$; $\Pi_{M^{Y(\omega)}}$ is the set of all
sequences of absolute probabilities for $M^{Y(\omega)}$; and $\mathbf{P}_{\pi^{Y(\omega)}}(\cdot)$
is the measure on $\mathcal{F}^{X}$ under which $(X_{n})_{n\in\mathbb{T}}$
is a Markov chain with transitions $M^{Y(\omega)}$ and absolute probabilities
$\pi^{Y(\omega)}$.

The condition (\ref{eq:cond_triviality_front}) can be interpreted
as meaning that the signal process is\emph{ conditionally ergodic}
given the observations, and is a key condition in studies of stability
with respect to initial conditions of nonlinear filters, see \citep{van2009stability}
and references therein. The condition (\ref{eq:equal_tail_front})
can be understood as meaning that any probability measure which makes
the signal process a Markov chain with transitions $M^{Y(\omega)}$
must have the same tail behavior as $\mathbf{P}(\cdot|\mathcal{F}^{Y})(\omega)$.

The remainder of the paper is structured as follows. Some notation
and other preliminaries are given in Section \ref{sec:Backward-Products-and}.
Section \ref{sec:Perfect-sampling} reviews existing literature on
perfect sampling for nonhomogeneous chains, gives the details of the
coupling method and describes its connection to weak ergodicity. Section
\ref{sec:Tail-Triviality-and} addresses tail-triviality. Section
\ref{sec:HMM} addresses the HMM setup. In Section \ref{sec:Discussion}
we discuss examples of HMM's for which our sampling method is not
successful, either through failure of (\ref{eq:cond_triviality_front})
or (\ref{eq:equal_tail_front}). We provide verifiable sufficient
conditions for successful coupling, discuss sampling when only finitely
many observations are available, and numerically illustrate how the
coupling can be influenced by the observation sequences. We discuss
an approach to the simulation of multiple dependent samples using
a single run of the perfect simulation and numerically investigate
its computational efficiency.

\section{Preliminaries\label{sec:Backward-Products-and}}

Throughout the paper, $E=\{1,\dots,s\}$ is a finite set, which we
endow with the discrete topology, and the corresponding Borel $\sigma$-algebra,
i.e.\ the power set of $E$, is denoted by $\mathcal{B}(E)$. For
any two probability distributions $\mu,\nu$ on $E$ we write the
total variation distance as
\begin{eqnarray*}
\left\Vert \mu-\nu\right\Vert  & := & \sup_{A\subset E}|\mu(A)-\nu(A)|\\
 & = & \frac{1}{2}\sum_{x\in E}\left|\mu(x)-\nu(x)\right|,
\end{eqnarray*}
and for a Markov kernel $K$ on $E$ we write Dobrushin's coefficient
\begin{eqnarray}
\beta(K) & := & \max_{(x,x^{\prime})\in E^{2}}\left\Vert K(x,\cdot)-K(x^{\prime},\cdot)\right\Vert ,\nonumber \\
 & = & 1-\min_{(x,x^{\prime})\in E^{2}}\sum_{z\in E}\min\{K(x,z),K(x^{\prime},z)\}.\label{eq:dobrushin_alt}
\end{eqnarray}

Throughout Sections~\ref{sec:Backward-Products-and}--\ref{sec:Tail-Triviality-and},
we fix an arbitrary collection of Markov kernels $M=(M_{n})_{n\in\mathbb{T}}$
on $E$ and we add slightly to the definitions of (\ref{eq:recurse M})
the convention that $M_{n,k}=Id$ whenever $k\leq n$, $n\in\mathbb{T}$.

We shall make extensive use of the following proposition, which expands
on Fact~\ref{fact:weak_ergodicity}, providing characterizations
of weak ergodicity in terms of Dobrushin's coefficient.
\begin{prop}
\noindent \label{prop:weak_ergo_charac}The following are equivalent.\vspace{0.15cm}

\noindent 1. for all $k\in\mathbb{T}$ and $(x,x^{\prime},z)\in E^{3}$,
$\lim_{n\rightarrow-\infty}M_{n,k}(x,z)-M_{n,k}(x^{\prime},z)=0$,

\noindent 2. $\text{card}(\Pi_{M})=1$,

\noindent 3. for all $k\in\mathbb{T}$, $\lim_{n\rightarrow-\infty}\beta(M_{n,k})=0$,

\noindent 4. there exists a strictly decreasing subsequence $\left(n_{i}\right)_{i\in\mathbb{N}}$
of $\mathbb{T}$ such that
\[
\sum_{i=0}^{\infty}1-\beta(M_{n_{i+1},n_{i}})=\infty,
\]
 \vspace{0.15cm}

\noindent and when any (and then all) of conditions 1--4 hold,
\begin{equation}
\lim_{n\rightarrow-\infty}M_{n,k}(x,z)-\pi_{k}(z)=0,\quad\quad\forall\; k\in\mathbb{T},\;(z,x)\in E^{2},\label{eq:M_conv_to_pi}
\end{equation}

\noindent where $\pi=(\pi_{n})_{n\in\mathbb{T}}$ is the unique sequence
of absolute probabilities for $M$.
\end{prop}
For proof of (1)$\Leftrightarrow$(2) and (\ref{eq:M_conv_to_pi})
see \citet[Theorem 4.20]{seneta2006non}, (1)$\Leftrightarrow$(3)
is immediate from the definition of $\beta(\cdot)$ and for proof
of (1)$\Leftrightarrow$(4) see \citet[Theorem 4.18]{seneta2006non}.

To connect with the perhaps more familiar case of homogeneous chains,
consider the case $E=\{0,1\}$ and $M_{n}(x,1-x)=M(x,1-x)=1$. For
all $\alpha\in(0,1)$, $\mbox{\ensuremath{\pi}}_{2n}(0)=\alpha=1-\mbox{\ensuremath{\pi}}_{2n}(1)$,
$\mbox{\ensuremath{\pi}}_{2n-1}(0)=1-\alpha=1-\mbox{\ensuremath{\pi}}_{2n-1}(1)$,
$n\in\mathbb{T}$ is a sequence of absolute probabilities, so that
there are infinitely many sequences of absolute probabilities for
$M$, even though there is a unique stationary distribution.

\section{Perfect sampling for nonhomogeneous chains\label{sec:Perfect-sampling}}

\subsection{Background}

We now review the existing literature on perfect sampling for nonhomogeneous
Markov chains. \citet[Section 5]{glynn2001two} formulated a perfect
sampling algorithm for a finite state-space chain conditioned to remain
in some set over a given time window and termination of their algorithm
in a.s. finite time follows from assumptions they make about the conditioned
process. \citet{stenflo2007perfect} devised a perfect sampling procedure
for nonhomogeneous backward products of stochastic matrices and showed
that, when there exists a constant $c>0$ such that
\begin{equation}
\inf_{n\in\mathbb{T}}\sum_{x^{\prime}\in E}\min_{x\in E}M_{n}(x,x^{\prime})\geq c,\label{eq:stenflo}
\end{equation}
the limit $\lim_{n\rightarrow-\infty}M_{n,0}(x,\cdot)$ exists, is
independent of $x$, and defines a probability distribution on $E$,
from which the algorithm of \citet{stenflo2007perfect} produces a
sample.

When (\ref{eq:stenflo}) holds, straightforward calculations show
that $c\leq1$ and by (\ref{eq:dobrushin_alt}), $\sup_{n\in\mathbb{T}}\beta(M_{n})\leq1-c$,
so part 4. of Proposition~\ref{prop:weak_ergo_charac} holds and
$\lim_{n\rightarrow-\infty}M_{n,0}(x,\cdot)$ is of course a member
of the unique sequence of absolute probabilities. However, part 4.\ of
Proposition~\ref{prop:weak_ergo_charac} is clearly a weaker condition
than (\ref{eq:stenflo}).

\subsection{The coupling}

Consider $\Omega^{\xi}:=(E^{s})^{\mathbb{T}}$ with product $\sigma$-algebra
$\mathcal{F}^{\xi}:=(\mathcal{B}(E)^{\otimes s})^{\otimes\mathbb{T}}$.
Define the coordinate process $\xi=(\xi_{n}^{x};x\in E,n\in\mathbb{T})$,
$\xi_{n}^{x}:\Omega^{\xi}\rightarrow E$, and let $\mathbf{Q}$ be
the probability measure on $(\Omega^{\xi},\mathcal{F}^{\xi})$,
\[
\mathbf{Q}(d\xi):=\bigotimes_{n\in\mathbb{T}}\bigotimes_{x\in E}M_{n}(x,\xi_{n}^{x})d\xi_{n}^{x},
\]
where $d\xi_{n}^{x}$ is counting measure on $E$, so that
\begin{eqnarray}
 &  & (\xi_{n}^{x};x\in E,n\in\mathbb{T})\text{ are independent under }\mathbf{Q},\label{eq:P_alg_1}\\
 &  & \mathbf{Q}(\xi_{n}^{x}=x^{\prime})=M_{n}(x,x^{\prime}),\quad(n,x,x^{\prime})\in\mathbb{T}\times E^{2}.\label{eq:P_alg_2}
\end{eqnarray}

For each $n\in\mathbb{T}$, define the random map
\begin{equation}
\Phi_{n}:x\in E\longmapsto\Phi_{n}(x):=\xi_{n}^{x}\in E,\label{eq:random_map_defn}
\end{equation}
and the compositions
\begin{equation}
\Phi_{n,k}:=\Phi_{k}\circ\Phi_{k-1}\circ\cdots\circ\Phi_{n+1},\quad n<k\in\mathbb{T},\label{eq:random_map_comp}
\end{equation}
so for example, $\Phi_{n,n+2}(x)=\Phi_{n+2}(\Phi_{n+1}(x))=\Phi_{n+2}(\xi_{n+1}^{x})=\xi_{n+2}^{\xi_{n+1}^{x}}$,
etc., and it is easily checked that
\begin{equation}
\mathbf{Q}(\Phi_{n}(x)=x^{\prime})=M_{n}(x,x^{\prime})\quad\text{ and }\quad\mathbf{Q}(\Phi_{n,k}(x)=x^{\prime})=M_{n,k}(x,x^{\prime}).\label{eq:Phi_equals_M}
\end{equation}
Now define the $\{-\infty\}\cup\mathbb{T}$-valued coalescence times
\begin{equation}
T_{k}:=\sup\{n<k\;:\;\text{image of }\Phi_{n,k}\text{ is a singleton}\},\quad k\in\mathbb{T},\label{eq:coalescence_times_defn}
\end{equation}
with $T_{k}:=-\infty$ when the set $\{n<k\;:\;\text{image of }\Phi_{n,k}\text{ is a singleton}\}$
is empty.
\begin{rem}
Note that in the time-homogeneous case, $M_{n}=M_{0}$ for all $n\in\mathbb{T}$,
the coalescence times $T_{n}$ are identically distributed and the
random maps $\Phi_{n}$ are iid.
\end{rem}
The main result of this section is the following theorem.
\begin{thm}
\label{thm:perfect sampling-1}Any (and then all) of Proposition~\ref{prop:weak_ergo_charac}
conditions 1--4 hold if and only if the coupling is successful, meaning
that for all $n\in\mathbb{T}$,
\begin{equation}
\mathbf{Q}(T_{n}>-\infty)=1,\quad\forall n\in\mathbb{T}.\label{eq:P_T_all_finite}
\end{equation}
Furthermore, if (\ref{eq:P_T_all_finite}) holds then, $\mathbf{Q}(\Phi_{T_{n},n}(x)\in\cdot)=\pi_{n}(\cdot)$
for all $x\in E$ and $n\in\mathbb{T}$, where $(\pi_{n})_{n\in\mathbb{T}}$
is the unique sequence of absolute probabilities for $M$. \end{thm}
\begin{rem}
As in the case of CFTP for time-homogeneous chains, if instead of
(\ref{eq:P_alg_1}) one allows dependence between $(\xi_{n}^{x};x\in E)$,
then it is possible to construct $M$ and $\mbox{\ensuremath{\mathbf{Q}}}$
such that the backward products of $M$ are weakly ergodic, but $\mathbf{Q}(T_{n}=-\infty)=1$
for all $n$, see e.g. \citep[Ch. 10]{haggstrom2002finite}. On the
other hand, under (\ref{eq:P_alg_1}), the situation is more clear-cut,
in the sense that the ``if and only if'' part of Theorem~\ref{thm:perfect sampling-1}
holds. However it should be noted that couplings involving dependence
between $(\xi_{n}^{x};x\in E)$ may lead to more computationally efficient
algorithms in some situations, especially when the number of states
$s$ is large.
\end{rem}
The proof of Theorem~\ref{thm:perfect sampling-1} is composed of
Propositions~\ref{prop:coupling_prop-->} and~\ref{prop:coupling_prop-<-},
which follow Lemma~\ref{lem:as_good_as_independent}.
\begin{lem}
\label{lem:as_good_as_independent}If for some $x^{\star}\in E$,
$\min_{x\in E}M_{n,k}(x,x^{*})\geq\epsilon>0$, then $\mathbf{Q}(T_{k}\geq n)\geq\epsilon^{s}$.\end{lem}
\begin{proof}
We have, with $A_{n,k}(j):=\{\Phi_{n,k}(1)=x^{*},\ldots,\Phi_{n,k}(j)=x^{*}\}$,
\[
\mathbf{Q}(T_{k}\geq n)\geq\mathbf{Q}(\Phi_{n,k}(1)=x^{*},\ldots,\Phi_{n,k}(s)=x^{*})=\mathbf{Q}(A_{n,k}(s)).
\]
We shall prove by an inductive argument that $\mathbf{Q}(A_{n,k}(s))\geq\epsilon^{s}$,
the inductive hypothesis being that, with $j\in\{2,\ldots,s\}$,
\begin{equation}
\mathbf{Q}(A_{n,k}(j-1))\geq\epsilon^{j-1},\label{eq:inductive_hypothesis}
\end{equation}
which is validated in the case $j=2$ by the assumption of the Lemma.
Since $\mathbf{Q}(A_{n,k}(j)\mid A_{n,k}(j-1))=\mathbf{Q}(\Phi_{n,k}(j)=x^{*}\mid A_{n,k}(j-1))$,
to show that (\ref{eq:inductive_hypothesis}) holds with $j-1$ replaced
by $j$, it is enough to establish
\begin{equation}
\mathbf{Q}(\Phi_{n,k}(j)\neq x^{*}\mid A_{n,k}(j-1))\leq1-\epsilon.\label{eq:lemma_1_to_prove}
\end{equation}
To this end, we need more some notation. Define $Q_{n,k}:E\times\left(2^{E}\right)^{k-n}\rightarrow[0,1]$
as
\begin{eqnarray*}
 &  & Q_{n,k}(x,S_{n+1},\ldots,S_{k})\\
 &  & :=\mathbf{Q}(\Phi_{n+1}(x)\notin S_{n+1},\Phi_{n,n+2}(x)\notin S_{n+2},\ldots,\Phi_{n,k}(x)\notin S_{k})\\
 &  & =\sum_{(x_{n+1},\ldots,x_{k})\in S_{n+1}^{\complement}\times\cdots\times S_{k}^{\complement}}M_{n+1}(x,x_{n+1})\prod_{i=n+2}^{k}M_{i}(x_{i-1},x_{i}),
\end{eqnarray*}
where the final equality is easily deduced from (\ref{eq:P_alg_2}).
Thus $Q_{n,k}(x,S_{n+1},\ldots,S_{k})$ is the probability that a
Markov chain evolving according to $M$ from time $n$ to $k$ starting
at $x$ avoids, for each $i\in\{n+1,\ldots,k\}$, the set $S_{i}$
at time $i$. It follows from the non-negativity of each $M_{i}$
that for any sequence of subsets $S_{n+1},\ldots,S_{k-1}$ of $E$
and any points $x,x^{\prime}\in E$,
\begin{equation}
Q_{n,k}(x,S_{n+1},\ldots,S_{k-1},\{x^{\prime}\})\leq Q_{n,k}(x,\emptyset,\ldots,\emptyset,\{x^{\prime}\})=\mathbf{Q}(\Phi_{n,k}(x)\neq x^{\prime}).\label{eq:Q_nk_upper_bound}
\end{equation}
Now with $j\in\{2,\ldots,s\}$ as in (\ref{eq:lemma_1_to_prove}),
introduce the notation
\begin{eqnarray*}
 &  & {\bf x}_{n,k,j}:=(x_{n+1,1},\ldots,x_{n+1,j-1},\ldots,x_{k-1,1},\ldots,x_{k-1,j-1}),
\end{eqnarray*}
which is a point in $E^{(k-n-1)(j-1)}=:E_{n,k,j}$, and
\[
B({\bf x}_{n,k,j}):=\left\{ \Phi_{n+1}(1)=x_{n+1,1},\ldots,\Phi_{n,k-1}(j-1)=x_{k-1,j-1}\right\} .
\]
Using (\ref{eq:P_alg_1}), (\ref{eq:P_alg_2}) and the fact that for
any $x,x^{\prime}\in E$ and $m\in\{n+1,\ldots,k\}$, $\{\Phi_{n,m}(x)=\Phi_{n,m}(x^{\prime})\}\subset\{\Phi_{n,k}(x)=\Phi_{n,k}(x^{\prime})\}$,
it follows by some elementary but tedious manipulations that
\begin{eqnarray*}
 &  & \mathbf{Q}\left(\Phi_{n,k}(j)\neq x^{*}|B({\bf x}_{n,k,j})\cap A_{n,k}(j-1)\right)\\
 &  & =Q_{n,k}\left(j,\bigcup_{i=1}^{j-1}\{x_{n+1,i}\},\ldots,\bigcup_{i=1}^{j-1}\{x_{k-1,i}\},\{x^{*}\}\right),
\end{eqnarray*}
so,
\begin{eqnarray}
 &  & \mathbf{Q}(\Phi_{n,k}(j)\neq x^{*}\mid A_{n,k}(j-1))\nonumber \\
 &  & =\sum_{{\bf x}_{n,k,j}\in E_{n,k,j}}Q_{n,k}\left(j,\bigcup_{i=1}^{j-1}\{x_{n+1,i}\},\ldots,\bigcup_{i=1}^{j-1}\{x_{k-1,i}\},\{x^{*}\}\right)\nonumber \\
 &  & \quad\quad\quad\quad\quad\quad\quad\times\mathbf{Q}\left(B({\bf x}_{n,k,j})\mid A_{n,k}(j-1)\right)\nonumber \\
 &  & \le Q_{n,k}\left(j,\emptyset,\ldots,\emptyset,\{x^{*}\}\right)=\mathbf{Q}(\Phi_{n,k}(j)\neq x^{*})\leq1-\epsilon,\label{eq:P_Q_P_line}
\end{eqnarray}
where the penultimate inequality and final equality are from (\ref{eq:Q_nk_upper_bound}),
and the final inequality follows from the hypothesis of the Lemma.
Thus the inductive hypothesis (\ref{eq:inductive_hypothesis}) holds
with $j-1$ replaced with $j$, and the proof of the Lemma is complete.\end{proof}
\begin{prop}
\label{prop:coupling_prop-->}If any of Proposition~\ref{prop:weak_ergo_charac}'s
conditions 1--4 hold, then for all $n\in\mathbb{T}$, $\mathbf{Q}(T_{n}>-\infty)=1$.\end{prop}
\begin{proof}
Under the hypothesis of the proposition, there is only one member
of $\Pi_{M}$, denote it by $\pi=(\pi_{n})_{n\in\mathbb{T}}$. Since
the $(\pi_{n})_{n\in\mathbb{T}}$ are probability distributions, for
each $n$ there must exist some $x_{n}^{\star}\in E$ such that $\pi_{n}(x_{n}^{\star})\geq s^{-1}$
(recall $E=\{1,\ldots,s\}$). Now fix $\epsilon\in(0,s^{-1})$. By
(\ref{eq:M_conv_to_pi}), for each $k\in\mathbb{T}$ there exists
$n<k$ such that
\[
M_{n,k}(x,x_{k}^{\star})\geq\pi_{k}(x_{k}^{\star})-(s^{-1}-\epsilon)\geq\epsilon>0,\quad\forall x\in E.
\]
We may then define $(k_{i})_{i\in\mathbb{N}}$ a strictly decreasing
subsequence of $\mathbb{T}$, with $k_{0}:=0$ and
\[
k_{i+1}:=\sup\left\{ n<k_{i}:M_{n,k_{i}}(x,x_{k_{i}}^{\star})\geq\epsilon\;\forall x\in E\right\} ,
\]
so that by construction,
\[
\inf_{i\in\mathbb{N}}\min_{x\in E}M_{k_{i+1},k_{i}}(x,x_{k_{i}}^{\star})\geq\epsilon>0.
\]
Lemma~\ref{lem:as_good_as_independent} then gives
\begin{equation}
\sup_{i\in\mathbb{N}}\mathbf{Q}(T_{k_{i}}\leq k_{i+1})\leq1-\epsilon^{s}.\label{eq:P_T_upper_bound}
\end{equation}
We now wish to apply this bound to control the tails of the coalescence
times $T_{n}$. To this end, first note that for any $n,k,k^{\prime}\in\mathbb{T}$,
$k^{\prime}<k<n$,
\[
\{T_{n}\geq k\}\cup\{T_{k}\geq k^{\prime}\}\subseteq\{T_{n}\geq k^{\prime}\}
\]
and since the events $\left\{ T_{n}\geq k\right\} $ and $\left\{ T_{k}\geq k^{\prime}\right\} $
are independent, we have
\begin{equation}
\mathbf{Q}(T_{n}<k^{\prime})\leq\mathbf{Q}(T_{n}<k)\mathbf{Q}(T_{k}<k^{\prime}),\label{eq:submultiplicative-1-1}
\end{equation}
Now fix $n\in\mathbb{T}$. Since $(k_{i})_{i\in\mathbb{N}}$ is strictly
decreasing, there exists some $i(n)$ such that $k_{i(n)}<n$. Then
by repeated application of (\ref{eq:submultiplicative-1-1}), we find
that for any $\ell>i(n)$,
\begin{equation}
\mathbf{Q}(T_{n}<k_{\ell})\leq\mathbf{Q}(T_{n}<k_{i(n)})\prod_{j=i(n)}^{\ell-1}\mathbf{Q}(T_{k_{j}}<k_{j+1}).\label{eq:P_iterated_bound-1}
\end{equation}
Now (\ref{eq:P_T_upper_bound}) provides an upper bound for the $j$-indexed
terms in (\ref{eq:P_iterated_bound-1}), and then taking $\ell\rightarrow\infty$
we find $\mathbf{Q}(T_{n}>-\infty)=1$, which completes the proof
of the proposition. \end{proof}
\begin{rem}
If one has available quantitative convergence information in addition
to (\ref{eq:M_conv_to_pi}), then the inequalities (\ref{eq:submultiplicative-1-1}),
(\ref{eq:P_iterated_bound-1}) and Lemma~\ref{lem:as_good_as_independent}
could be used to bound the moments of the $T_{n}$.\end{rem}
\begin{prop}
\noindent \label{prop:coupling_prop-<-}If for all $n\in\mathbb{T}$,
$\mathbf{Q}(T_{n}>-\infty)=1$, then both of the following hold.\vspace{0.15cm}

\noindent 1. there exists a strictly decreasing subsequence $(n_{i})_{i\in\mathbb{N}}$
of $\mathbb{T}$ such that $\sum_{i=0}^{\infty}1-\beta(M_{n_{i+1},n_{i}})=\infty$,
i.e.,\ condition 4.\ of Proposition~\ref{prop:weak_ergo_charac}
holds.

\noindent 2. for all $n\in\mathbb{T}$, $\mathbf{Q}(X_{T_{n}}^{x}(n)\in\cdot)=\pi_{n}(\cdot)$
for all $x\in E$, where $(\pi_{n})_{n\in\mathbb{T}}$ is the unique
sequence of absolute probabilities for $M$. \vspace{0.15cm}
\end{prop}
\begin{proof}
Fix some $\delta>0$. Under the hypothesis, we have that for each
$n$ there exists $k\in\mathbb{T}$ such that $\mathbf{Q}(T_{n}<k)<\delta$.
We may therefore define $(k_{i})_{i\in\mathbb{N}}$ a strictly decreasing
subsequence of $\mathbb{T}$ with $k_{0}:=0$,

\[
k_{i+1}:=\sup\{n<k_{i}:\mathbf{Q}(T_{k_{i}}<n)<\delta\},
\]
so that by construction,
\begin{equation}
\sup_{i\in\mathbb{N}}\mathbf{Q}(T_{k_{i}}<k_{i+1})<\delta.\label{eq:P_sup_uniform-1}
\end{equation}
Now for any $x,x^{\prime}\in E$,  $\{T_{k_{i}}\geq k_{i+1}\}\subseteq\{\Phi_{k_{i+1},k_{i}}(x)=\Phi_{k_{i+1},k_{i}}(x^{\prime})\}$
so $\mathbf{Q}(\Phi_{k_{i+1},k_{i}}(x)\neq\Phi_{k_{i+1},k_{i}}(x^{\prime}))\leq\mathbf{Q}(T_{k_{i}}<k_{i+1})$.
Combining this observation with (\ref{eq:P_sup_uniform-1}) and (\ref{eq:Phi_equals_M}),
we obtain:
\begin{eqnarray}
\beta(M_{k_{i+1},k_{i}}) & = & \max_{x,x^{\prime}}\|M_{k_{i+1},k_{i}}(x,\cdot)-M_{k_{i+1},k_{i}}(x^{\prime},\cdot)\|\nonumber \\
 & = & \max_{x,x^{\prime}}\sup_{A\subset E}|\mathbf{Q}(\Phi_{k_{i+1},k_{i}}(x)\in A)-\mathbf{Q}(\Phi_{k_{i+1},k_{i}}(x^{\prime})\in A)|\nonumber \\
 & \leq & \max_{x,x^{\prime}}\mathbf{Q}(\Phi_{k_{i+1},k_{i}}(x)\neq\Phi_{k_{i+1},k_{i}}(x^{\prime}))<\delta,\quad\forall i\in\mathbb{N},\label{eq:beta(M)}
\end{eqnarray}
where we have used the fact that for any two $E$-valued random variables
$X,X^{\prime}$ defined on a common probability space, $\sup_{A\subset E}|\mathbf{P}(X\in A)-\mathbf{P}(X^{\prime}\in A)|\leq\mathbf{P}(X\neq X^{\prime})$
\citep[p. 12]{lindvall2002lectures}. From (\ref{eq:beta(M)}) we
immediately have $\sum_{i}1-\beta(M_{k_{i+1},k_{i}})=\infty$, which
completes the proof of part (1).

For part (2), fix any $n\in\mathbb{T}$, and note that on the event
$\{T_{n}>-\infty\}$, $\Phi_{T_{n},n}(x)$ is well defined as a random
variable. When $\mathbf{Q}(T_{n}>-\infty)=1$, we have by construction
of the algorithm that $\lim_{k\rightarrow-\infty}\Phi_{n+k,n}(x)=\Phi_{T_{n},n}(x)$,
$\mathbf{Q}$-a.s. Using (\ref{eq:M_conv_to_pi}) we also have for
any $z\in E$, $\mathbf{Q}(\Phi_{n+k,n}(x)=z)=M_{n+k,n}(x,z)\rightarrow\pi_{n}(z)$
as $k\rightarrow-\infty$, hence  $\mathbf{Q}(\Phi_{T_{n},n}(x)=z)=\pi_{n}(z)$.
The proof is complete.
\end{proof}

\section{Tail triviality and unicity of absolute probabilities\label{sec:Tail-Triviality-and}}

Let $\Omega^{X}=E^{\mathbb{T}}$, let $\mathcal{F}^{X}=\mathcal{B}(E)^{\otimes\mathbb{T}}$
be the product $\sigma$-algebra. Let $X=(X_{n})_{n\in\mathbb{T}}$
be the coordinate process on $\Omega^{X}$ and for $I\subset\mathbb{T}$,
define $\mathcal{F}_{I}^{X}=\sigma(X_{n};n\in I)$. As in Section
\ref{sec:Introduction}, for any $\pi\in\Pi_{M}$ we let $\mathbf{P}_{\pi}$
be the probability measure on $(\Omega^{X},\mathcal{F}^{X})$ constructed
from the finite dimensional distributions $(\mathbf{P}_{\pi}^{(n)})_{n\in\mathbb{T}}$
given by
\begin{equation}
\mathbf{P}_{\pi}^{(n)}(X_{n}=x_{n},\ldots,X_{0}=x_{0}):=\pi_{n}(x_{n})\prod_{k=n+1}^{0}M_{k}(x_{k-1},x_{k}).\label{eq:P_pi_fidi}
\end{equation}
Expectation w.r.t. $\mathbf{P}_{\pi}$ is denoted by $\mathbf{E}_{\pi}$.
The main result of this section is the following theorem, which via
Proposition~\ref{prop:weak_ergo_charac} gives an alternative characterization
of the success of the coupling in the sense of Theorem~\ref{thm:perfect sampling-1}.
\begin{thm}
\noindent \label{thm:tail}The following are equivalent.\vspace{0.15cm}

\noindent 1. $\text{card}(\Pi_{M})=1$.

\noindent 2. $\mathbf{P}_{\pi}(A)=\mathbf{P}_{\pi}(A)^{2}=\mathbf{P}_{\tilde{\pi}}(A),\quad\forall(\pi,\tilde{\pi},A)\in\Pi_{M}\times\Pi_{M}\times\bigcap_{n\in\mathbb{T}}\mathcal{F}_{]-\infty,n]}^{X}$.
\vspace{0.15cm}

\end{thm}
When 1. holds then obviously $\mathbf{P}_{\pi}(A)=\mathbf{P}_{\tilde{\pi}}(A)$.
The proof of 1.$\Rightarrow$2. is completed by Proposition~\ref{prop:non_triv_implies_card>1}.
The implication 2.$\Rightarrow$1. is the subject of Proposition~\ref{prop:triv_imples_card_1}.
\begin{prop}
\label{prop:non_triv_implies_card>1} If there exists $\pi\in\Pi_{M}$
and $A\in\bigcap_{n\in\mathbb{T}}\mathcal{F}_{]-\infty,n]}^{X}$ such
that $\mathbf{P}_{\pi}(A)\in]0,1[$, then $\text{card}(\Pi_{M})>1$.\end{prop}
\begin{proof}
Let $\pi$ and $A$ be as in the statement of the proposition and
with $Z(\omega):=\mathbb{I}_{A}(\omega)/\mathbf{P}_{\pi}(A)$, define
a new probability measure $\tilde{\mathbf{P}}$ on $(\Omega^{X},\mathcal{F}^{X})$
by $\tilde{\mathbf{P}}(d\omega):=Z(\omega)\mathbf{P}_{\pi}(d\omega)$,
i.e., $\tilde{\mathbf{P}}(\cdot)=\mathbf{P}_{\pi}(\cdot\,|A)$. Define
also the sequence of marginal distributions $\tilde{\pi}=(\tilde{\pi}_{n})_{n\in\mathbb{T}}$,
$\tilde{\pi}_{n}(\cdot):=\tilde{\mathbf{P}}(X_{n}\in\cdot\,)$. We
are going to show that $\tilde{\pi}\in\Pi_{M}$ and $\tilde{\pi}\neq\pi$,
thus proving $\text{card}(\Pi_{M})>1$ as desired.

The Markov property of $X$ under $\mathbf{P}_{\pi}$ and the fact
that $Z$ is measurable w.r.t.\ to $\bigcap_{n}\mathcal{F}_{]-\infty,n]}^{X}$
combine to give $\mathbf{E}_{\pi}[Z|\mathcal{F}_{[n,0]}^{X}]=\mathbf{E}_{\pi}[Z|\sigma(X_{n})]$,
$\mathbf{P}$-a.s., so, for each $n\in\mathbb{T}$, there exists a
measurable function $h_{n}$ on $E$, uniquely defined and nonnegative
$\pi_{n}$-almost everywhere, such that $h_{n}(X_{n})=\mathbf{E}_{\pi}[Z|\mathcal{F}_{[n,0]}^{X}]$,
$\mathbf{P}$-a.s.  We then have, for any $n\in\mathbb{T}$ and $(x_{n},\ldots,x_{0})\in E^{|n|+1}$,
\begin{eqnarray}
 &  & \tilde{\mathbf{P}}\left(\{\omega:X_{n}(\omega)=x_{n},\ldots,X_{0}(\omega)=x_{0}\}\right)\nonumber \\
 &  & =\int_{\{\omega:X_{n}(\omega)=x_{n},\ldots,X_{0}(\omega)=x_{0}\}}Z(\omega)d\mathbb{\mathbf{P}}_{\pi}\nonumber \\
 &  & =\int_{\{\omega:X_{n}(\omega)=x_{n},\ldots,X_{0}(\omega)=x_{0}\}}\mathbf{E}_{\pi}[Z|\mathcal{F}_{[n,0]}^{X}](\omega)d\mathbb{\mathbf{P}}_{\pi}\nonumber \\
 &  & =\int_{\{\omega:X_{n}(\omega)=x_{n},\ldots,X_{0}(\omega)=x_{0}\}}h_{n}(X_{n}(\omega))d\mathbb{\mathbf{P}}_{\pi}\nonumber \\
 &  & =h_{n}(x_{n})\pi_{n}(x_{n})\prod_{k=n+1}^{0}M_{k}(x_{k-1},x_{k}).\label{eq:P_tilde_fidi}
\end{eqnarray}
From (\ref{eq:P_tilde_fidi}) we immediately deduce three facts. Firstly,
for each $n\in\mathbb{T}$ and $x\in E$, $\tilde{\pi}_{n}(x)=h_{n}(x)\pi_{n}(x)$.
Secondly, $\tilde{\pi}\in\Pi_{M}$. Thirdly, the finite dimensional
marginals of $\tilde{\mathbf{P}}$ coincide with those of $\mathbf{P}_{\tilde{\pi}}$,
so by a monotone class argument, $\tilde{\mathbf{P}}=\mathbf{P}_{\tilde{\pi}}$.

It remains to prove that $\tilde{\pi}\neq\pi$. Yet again by a monotone
class argument, note that for any $\mu,\nu\in\Pi_{M}$, if for all
$n\in\mathbb{T}$, $\mathbf{P}_{\mu}^{(n)}=\mathbf{P}_{\nu}^{(n)}$
, then $\mathbf{P}_{\mu}=\mathbf{P}_{\nu}$. We have already seen
that $\tilde{\mathbf{P}}=\mathbf{P}_{\tilde{\pi}}$, and by construction,
$\tilde{\mathbf{P}}\neq\mathbf{P}_{\pi}$, so by applying the contrapositive
of the implication in the previous sentence, there must exist some
$n\in\mathbb{T}$ such that $\mathbf{P}_{\pi}^{(n)}\neq\mathbf{P}_{\tilde{\pi}}^{(n)}$,
which is only possible if there exists some $x$ such that $\pi_{n}(x)\neq\tilde{\pi}_{n}(x)$.
This completes the proof. \end{proof}
\begin{prop}
\label{prop:triv_imples_card_1} If for all $(\pi,\tilde{\pi},A)\in\Pi_{M}\times\Pi_{M}\times\bigcap_{n\in\mathbb{T}}\mathcal{F}_{]-\infty,n]}^{X}$,
$\mathbf{P}_{\pi}(A)=\mathbf{P}_{\pi}(A)^{2}=\mathbf{P}_{\tilde{\pi}}(A)$,
then $\text{card}(\Pi_{M})=1$.\end{prop}
\begin{proof}
Fix arbitrarily $x\in E$, $k\in\mathbb{T}$ and let $\pi$ be any
member of $\Pi_{M}$. For any $n\leq k$, we have $\mathbf{P}_{\pi}(X_{k}=x|\mathcal{F}_{]-\infty,n]}^{X})=M_{n,k}(X_{n},x)$,
$\mathbf{P}_{\pi}$-a.s.\ and since $\mathcal{F}_{]-\infty,n]}^{X}\searrow\bigcap_{n\in\mathbb{T}}\mathcal{F}_{]-\infty,n]}^{X}$,
a classical martingale convergence theorem \citep[p. 331, Theorem 4.3]{doob1953stochastic}
dictates that $\lim_{n\rightarrow-\infty}\mathbf{P}_{\pi}(X_{k}=x|\mathcal{F}_{]-\infty,n]}^{X})=\mathbf{P}_{\pi}(X_{k}=x|\bigcap_{n\in\mathbb{T}}\mathcal{F}_{]-\infty,n]}^{X})$,
$\mathbf{P}_{\pi}$-a.s. Under the hypothesis of the proposition,
$\mathbf{P}_{\pi}(A)\in\{0,1\}$ for all $A\in\bigcap_{n\in\mathbb{T}}\mathcal{F}_{]-\infty,n]}^{X}$,
and by construction $\mathbf{P}_{\pi}(X_{k}=x)=\pi_{k}(x)$, so we
obtain:
\begin{equation}
\lim_{n\rightarrow-\infty}M_{n,k}(X_{n},x)=\pi_{k}(x),\quad\mathbf{P}_{\pi}-\text{a.s.}\label{eq:M_to_mu}
\end{equation}

Now choose any $\tilde{\pi}\in\Pi_{M}$. Repeating the above argument
we obtain
\begin{equation}
\lim_{n\rightarrow-\infty}M_{n,k}(X_{n},x)=\tilde{\pi}_{k}(x),\quad\mathbf{P}_{\tilde{\pi}}-\text{a.s.},\label{eq:M_to_pi}
\end{equation}
and since $A_{x}:=\{\lim_{n\rightarrow-\infty}M_{n,k}(X_{n},x)=\pi_{k}(x)\}\in\bigcap_{n\in\mathbb{T}}\mathcal{F}_{]-\infty,n]}^{X}$,
the hypothesis of the proposition dictates $\mathbf{P}_{\pi}(A_{x})=\mathbf{P}_{\tilde{\pi}}(A_{x})$,
so from (\ref{eq:M_to_mu}) and (\ref{eq:M_to_pi}) we find $\pi_{k}(x)=\tilde{\pi}_{k}(x)$.
Since $x$ and $k$ were arbitrary, we have thus established $\pi=\tilde{\pi}$,
and since $\pi$ and $\tilde{\pi}$ were arbitrary members of $\Pi_{M}$
we have proved that $\text{card}(\Pi_{M})=1$.
\end{proof}

\section{Perfect sampling for hidden Markov models\label{sec:HMM}}

\subsection{The model\label{sub:The-model-hmm}}

Throughout Section~\ref{sec:HMM}, we take $\Omega^{X}:=E^{\mathbb{T}}$
equipped with the product $\sigma$-algebra $\mathcal{B}(E)^{\otimes\mathbb{T}}$,
we introduce $F$ a non-empty, Polish state-space with Borel $\sigma$-algebra
denoted by $\mathcal{B}(F)$, and we consider $\Omega^{Y}:=F^{\mathbb{T}}$
equipped with the product $\sigma$-algebra $\mathcal{B}(F)^{\otimes\mathbb{T}}$.
Define $\Omega:=\Omega^{X}\times\Omega^{Y}$ and the coordinate projections:
$\zeta,\eta$ by
\[
\zeta:(x,y)\in\Omega\mapsto x\in\Omega^{X},\quad\eta:(x,y)\in\Omega\mapsto y\in\Omega^{Y},
\]
 and $(\tilde{X}_{n})_{n\in\mathbb{T}}$, $(\tilde{Y}_{n})_{n\in\mathbb{T}}$
by
\begin{eqnarray*}
 &  & \tilde{X}_{n}:x=(\ldots,x_{-1},x_{0})\in\Omega^{X}\mapsto x_{n}\in E,\\
 &  & \tilde{Y}_{n}:y=(\ldots,y_{-1},y_{0})\in\Omega^{Y}\mapsto y_{n}\in F.
\end{eqnarray*}
 Then let
\[
X_{n}:=\tilde{X}_{n}\circ\zeta,\quad Y_{n}:=\tilde{Y}_{n}\circ\eta,
\]
so clearly $X_{n}:\omega=(x,y)\in\Omega\mapsto x_{n}\in E$ and $Y_{n}:\omega=(x,y)\in\Omega\mapsto y_{n}\in F$.
We shall write $X$ and $Y$ for respectively the $E^{\mathbb{T}}$
and $F^{\mathbb{T}}$-valued random variables $(X_{n})_{n\in\mathbb{T}}$
and $(Y_{n})_{n\in\mathbb{T}}$.

Let $\mathcal{F}$ be the Borel $\sigma$-algebra on $\Omega$ and
for $I\subset\mathbb{T}$, define
\begin{eqnarray*}
\mathcal{F}_{I}^{\tilde{X}}:=\sigma(\tilde{X}_{n};n\in I), &  & \mathcal{F}_{I}^{X}:=\sigma(X_{n};n\in I),\\
\mathcal{F}_{I}^{\tilde{Y}}:=\sigma(\tilde{Y}_{n};n\in I), &  & \mathcal{F}_{I}^{Y}:=\sigma(Y_{n};n\in I),
\end{eqnarray*}
\[
\mathcal{F}_{I}:=\mathcal{F}_{I}^{X}\vee\mathcal{F}_{I}^{Y},
\]
and $\mathcal{F}^{\tilde{X}}:=\mathcal{F}_{\mathbb{T}}^{\tilde{X}}$,
$\mathcal{F}^{X}:=\mathcal{F}_{\mathbb{T}}^{X}$, $\mathcal{F}^{\tilde{Y}}:=\mathcal{F}_{\mathbb{T}}^{\tilde{Y}}$,
$\mathcal{F}^{Y}:=\mathcal{F}_{\mathbb{T}}^{Y}$ .

Now introduce two sequences of probability kernels $M=(M_{n})_{n\in\mathbb{T}}$
and $G=(G_{n})_{n\in\mathbb{T}}$, with each $M_{n}:E\times\mathcal{B}(E)\rightarrow[0,1]$
and $G_{n}:E\times\mathcal{B}(F)\rightarrow[0,1]$. We assume that
$G_{n}(x,dy)=g_{n}(x,y)\psi(dy)$ for some $g_{n}:E\times F\rightarrow[0,\infty[$
and $\psi$ a $\sigma$- finite measure on $(F,\mathcal{B}(F))$.

Throughout Section~\ref{sec:HMM}, $\mathbf{P}$ is a probability
measure on $\left(\Omega,\mathcal{F}\right)$ under which $(X,Y)$
is a hidden Markov model, constructed as follows. Fix some $\pi=(\pi_{n})_{n\in\mathbb{T}}\in\Pi_{M}$.
For each $n\in\mathbb{T}$ define a probability $\mathbf{P}^{(n)}$
on $(\mathcal{B}(E)\otimes\mathcal{B}(F))^{\otimes(|n|+1)}$ by
\begin{equation}
\mbox{\ensuremath{\mathbf{P}}}^{(n)}(A)=\int_{A}\pi_{n}(dx_{n})G_{n}(x_{n},dy_{n})\prod_{k=n+1}^{0}M_{k}(x_{k-1},dx_{k})G_{k}(x_{k},dy_{k}),\label{eq:P_HMM_fidi}
\end{equation}
with the convention that the product is unity when $n=0$. Since $\pi\in\Pi_{M}$,
the $\mbox{\ensuremath{\mathbf{P}}}^{(n)}$ are consistent, giving
rise via the usual extension argument to a probability measure $\mathbf{P}$
on $(\Omega,\mathcal{F})$. Expectation w.r.t. $\mathbf{P}$ is denoted
by $\mathbf{E}$. We shall write $\mathbf{P}\circ Y^{-1}$ for the
pushforward of $\mathbf{P}$ by $Y$, i.e. $(\mathbf{P}\circ Y^{-1})(H)=\mathbf{P}(\{\omega\in\Omega:Y(\omega)\in H\})$,
for $H\in\mathcal{B}(F)^{\otimes\mathbb{T}}$.

Let us now remark upon some details of this setup (the analogues of
the following properties for an HMM on a nonnegative time horizon
are well known and the arguments involved in establishing them depend
only superficially on the direction of time). Under $\mathbf{P}$,
the bi-variate process $(X_{n},Y_{n})_{n\in\mathbb{T}}$ is Markov,
which implies that the following holds $\mathbf{P}-\text{a.s.}$,
\begin{equation}
\mathbf{P}(X_{n}\in\cdot\,|\mathcal{F}_{]-\infty,0]}^{Y}\vee\mathcal{F}_{]-\infty,n-1]}^{X})=\mathbf{P}(X_{n}\in\cdot\,|\mathcal{F}_{]-\infty,0]}^{Y}\vee\sigma(X_{n-1})).\label{eq:hmm_cond_ind_1}
\end{equation}
Moreover, under $\mathbf{P}$, $X$ is a Markov, with for each $n\in\mathbb{T}$,
$X_{n}$ distributed according to $\pi_{n}$ and $X_{n}|X_{n-1}\sim M_{n}(X_{n-1},\cdot\,)$.
The observations $Y$ are conditionally independent given $X$, and
the conditional distribution of $Y_{n}$ given $X$ is $G_{n}(X_{n},\cdot\,)$.
It follows from this conditional-independence structure that the following
holds $\mathbf{P}-\text{a.s.}$,
\begin{equation}
\mathbf{P}(X_{n}\in\cdot\,|\mathcal{F}_{]-\infty,0]}^{Y}\vee\sigma(X_{n-1}))=\mathbf{P}(X_{n}\in\cdot\,|\mathcal{F}_{[n,0]}^{Y}\vee\sigma(X_{n-1})).\label{eq:hmm_cond_ind_2}
\end{equation}

\subsection{Transition kernels of the conditional signal process}

Define the sequence of functions $(\phi_{n})_{n\in\mathbb{T}}$, each
$\phi_{n}:E\times F^{n}\rightarrow[0,\infty[$, recursively as
\begin{eqnarray}
 &  & \phi_{0}(x,y_{0}):=\sum_{x^{\prime}\in E}M_{0}(x,x^{\prime})g_{0}(x^{\prime},y_{0}),\nonumber \\
 &  & \phi_{n-1}(x,y_{n-1:0}):=\sum_{x^{\prime}\in E}M_{n-1}(x,x^{\prime})g_{n-1}(x^{\prime},y_{n-1})\phi_{n}(x^{\prime},y_{n:0}).\label{eq:phi_defn}
\end{eqnarray}

Now with $y=(y_{n})_{n\in\mathbb{T}}$, introduce for each $n\in\mathbb{T}$,
\begin{equation}
M_{n}^{y}(x,x^{\prime}):=\begin{cases}
\dfrac{M_{n}(x,x^{\prime})g_{n}(x^{\prime},y_{n})\phi_{n+1}(x^{\prime},y_{n+1:0})}{\phi_{n}(x,y_{n:0})}, & \;\phi_{n}(x,y_{n:0})>0,\\
M_{n}(x,x^{\prime}), & \;\phi_{n}(x,y_{n:0})=0,
\end{cases}\label{eq:M^y_defn}
\end{equation}
with the convention that $\phi_{1}(x,y_{1:0})\equiv1$. Similarly
to (\ref{eq:recurse M}), let
\begin{equation}
M_{k,k}^{y}:=Id,\quad\quad M_{n-1,k}^{y}(x,x^{\prime}):=\sum_{z\in E}M_{n}^{y}(x,z)M_{n,k}^{y}(z,x^{\prime}),\quad n\leq k.\label{eq:recurse My}
\end{equation}

According to (\ref{eq:M^y_defn}), for each $y$, $M_{n}^{y}(\cdot,\cdot)$
is clearly a Markov kernel on $E$. This kernel provides a version
of the conditional probabilities in (\ref{eq:hmm_cond_ind_1})--(\ref{eq:hmm_cond_ind_2}),
in the sense of the following lemma, whose proof is given in the appendix.
\begin{lem}
\label{lem:M^Y_is_a_version}For each $n\in\mathbb{T}$ and $x\in E$,
\[
\mathbf{P}(X_{n}=x|\mathcal{F}_{[n,0]}^{Y}\vee\sigma(X_{n-1}))=M_{n}^{Y}(X_{n-1},x),\quad\mathbf{P}-\text{a.s.}
\]

\end{lem}
We next establish the existence of a particular $y$-dependent sequence
of absolute probabilities for the Markov kernels $M^{y}=(M_{n}^{y})_{n\in\mathbb{T}}$.
\begin{lem}
\label{lem:existence_of_pi^y}For each $n\in\mathbb{T}$, there exists
a probability kernel $\mu_{n}^{\cdot}(\cdot):\Omega^{Y}\times E\rightarrow[0,1]$,
such that for all $x\in E$,
\[
\mathbf{P}(X_{n}=x|\mathcal{F}_{]-\infty,0]}^{Y})=\mu_{n}^{Y}(x),\quad\mathbf{P}-a.s.,
\]
and
\[
\sum_{x^{\prime}\in E}\mu_{n-1}^{Y}(x^{\prime})M_{n}^{Y}(x^{\prime},x)=\mu_{n}^{Y}(x),\quad\mathbf{P}-a.s.
\]
\end{lem}
\begin{proof}
Since $E$ is a finite set, the existence for any $n$ of a probability
kernel $\mu_{n}^{\cdot}(\cdot):\Omega^{Y}\times E\rightarrow[0,1]$
satisfying $\mathbf{P}(X_{n}=x|\mathcal{F}_{]-\infty,0]}^{Y})=\mu_{n}^{Y}(x)$,
$\mathbf{P}$-a.s. for all $x$, is immediate. Then by the tower property
of conditional expectation, Lemma~\ref{lem:M^Y_is_a_version} and
(\ref{eq:hmm_cond_ind_2}), the following equalities hold $\mathbf{P}$-a.s.,
\begin{eqnarray*}
\mu_{n}^{Y}(x) & = & \mathbf{P}(X_{n}=x|\mathcal{F}_{]-\infty,0]}^{Y})\\
 & = & \mathbf{E}[\mathbf{E}[\mathbb{I}\{X_{n}=x\}|\mathcal{F}_{]-\infty,0]}^{Y}\vee\sigma(X_{n-1})]|\mathcal{F}_{]-\infty,0]}^{Y}]\\
 & = & \mathbf{E}[M_{n}^{Y}(X_{n-1},x)|\mathcal{F}_{]-\infty,0]}^{Y}]\\
 & = & \sum_{x^{\prime}\in E}\mu_{n-1}^{Y}(x^{\prime})M_{n}^{Y}(x^{\prime},x).
\end{eqnarray*}

\end{proof}

\subsection{The coupling for the HMM\label{sub:The-coupling-for}}

Our next main objective is to apply the construction and results of
Section~\ref{sec:Perfect-sampling} to derive and study a perfect
sampling procedure associated with the Markov kernels $M^{y}$. The
setup is as follows. With $\Omega^{\xi}=(E^{s})^{\otimes\mathbb{T}}$
and $\mathcal{F}^{\xi}=(\mathcal{B}(E)^{\otimes s})^{\otimes\mathbb{T}}$,
let $\mathbf{Q}^{\cdot}(\cdot):\Omega^{Y}\times\mathcal{F}^{\xi}\rightarrow[0,1]$
be a probability kernel such that for each $y\in\Omega^{Y}$ the coordinate
projections $(\xi_{n}^{x};x\in E;n\in\mathbb{T})$ are distributed
under $\mathbf{Q}^{y}(\cdot)$ as
\begin{eqnarray}
 &  & (\xi_{n}^{x};x\in E,n\in\mathbb{T})\text{ are independent,}\label{eq:hmm_P_alg_1}\\
 &  & \mathbf{Q}^{y}(\xi_{n}^{x}=x^{\prime})=M_{n}^{y}(x,x^{\prime}),\quad(n,x,x^{\prime})\in\mathbb{T}\times E^{2}.\label{eq:hmm_P_alg_2}
\end{eqnarray}
 Thus with $y$ fixed, $\mathbf{Q}^{y}(\cdot)$ may be regarded as
an instance of the probability measure denoted $\mathbf{Q}(\cdot)$
in Section~\ref{sec:Perfect-sampling}. Also, let the maps $\Phi_{n}$,
$\Phi_{n,k}$ and the coalescence times $(T_{n})_{n\in\mathbb{T}}$
be defined exactly as in equations (\ref{eq:random_map_defn}), (\ref{eq:random_map_comp})
and (\ref{eq:coalescence_times_defn}) of Section~\ref{sec:Perfect-sampling}.
\begin{prop}
\label{prop:fixed_y_weak_ergo_hmm}Fix any $y\in\Omega^{Y}$. Any
(and all) of Proposition~\ref{prop:weak_ergo_charac} conditions
1.--4. hold for the Markov kernels $M^{y}$, if and only if
\begin{equation}
\mathbf{Q}^{y}(T_{n}>-\infty)=1,\quad\forall n\in\mathbb{T}.\label{eq:P_T_n_all_finite_hmm}
\end{equation}
Furthermore, if (\ref{eq:P_T_n_all_finite_hmm}) holds then for all
$n\in\mathbb{T}$, $\mathbf{Q}^{y}(\Phi_{T_{n},n}(x)\in\cdot)=\pi_{n}^{y}(\cdot)$
for all $x\in E$, where $\pi^{y}=(\pi_{n}^{y})_{n\in\mathbb{T}}$
is the unique sequence of absolute probabilities for $M^{y}$. If
(\ref{eq:P_T_n_all_finite_hmm}) holds for $y$ in a set of $\mathbf{P}\circ Y^{-1}$
probability $1$, then for all $n\in\mathbb{T}$ and $x\in E$, $\mathbf{Q}^{Y}(\Phi_{T_{n},n}(x)\in\cdot)=\mathbf{\mathbf{P}}(X_{n}\in\cdot|\mathcal{F}^{Y})$,
$\mathbf{P}$-a.s., and we call the coupling a.s. successful.\end{prop}
\begin{proof}
For fixed $y\in\Omega^{Y}$, the claimed equivalence between (\ref{eq:P_T_n_all_finite_hmm})
and Proposition~\ref{prop:weak_ergo_charac} conditions 1.--4. holding
for the Markov kernels $M^{y}$ is an application of Theorem~\ref{thm:perfect sampling-1}.
So too is the equality $\mathbf{Q}^{y}(\Phi_{T_{n},n}(x)\in\cdot)=\pi_{n}^{y}(\cdot)$.
Since $E$ is a finite set and $\mathbb{T}$ is countable, it follows
from Lemma~\ref{lem:existence_of_pi^y} that there exists $H\in\mathcal{F}^{Y}$
with $\mathbf{P}(H)=1$ and such that for all $\omega\in H$, $n\in\mathbb{T}$
and $x\in E$, $\mathbf{\mathbf{P}}(X_{n}=x|\mathcal{F}^{Y})(\omega)=\mu_{n}^{Y(\omega)}(x)$
and $\sum_{x^{\prime}\in E}\mu_{n-1}^{Y(\omega)}(x^{\prime})M_{n}^{Y(\omega)}(x^{\prime},x)=\mu_{n}^{Y(\omega)}(x)$.
If, as hypothesized in the statement, there exists $\tilde{H}\in\mathcal{F}^{\tilde{Y}}$
such that for all $y\in\tilde{H}$, $\mathbf{Q}^{y}(T_{n}>-\infty)=1$
for all $n\in\mathbb{T}$, then for all $\omega\in H\cap Y^{-1}(\tilde{H})$,
$M{}^{Y(\omega)}$ admits a unique sequence of absolute probabilities,
and so $\pi_{n}^{Y(\omega)}(x)=\mu_{n}^{Y(\omega)}(x)=\mathbf{\mathbf{P}}(X_{n}=x|\mathcal{F}^{Y})(\omega)$.
\end{proof}
We present in Algorithm~\ref{alg:Perfect-sampling-for-hmm} some
steps of the sampling procedure, in order to emphasize the way that
the observations enter into recursive computations. For simplicity
of presentation, we consider the case of implementing the coupling
until $T_{0}=\sup\{n<0:\text{ image of \ensuremath{\Phi_{n,0}}is a singleton}\}$,
thus upon termination in a.s. finite time of the below algorithm,
the output value is a sample from $\mu_{0}^{y}$. The important point
here is that to run this algorithm one needs access to only the observations
$y_{0},\ldots,y_{T_{0}}$.

\begin{algorithm}[h]
\begin{raggedright}
for each $x\in E$, set
\[
\phi_{0}(x,y_{0})=\sum_{x^{\prime}\in E}M_{0}(x,x^{\prime})g_{0}(x^{\prime},y_{0}),
\]
\quad{}and by convention, $\phi_{1}(x,y_{1:0})=1$.
\par\end{raggedright}

\begin{raggedright}
set $\Phi_{0,0}=Id$
\par\end{raggedright}

\begin{raggedright}
set $n=0$
\par\end{raggedright}

\begin{raggedright}
while $card(\text{image of }\Phi_{n,0})>1$
\par\end{raggedright}

\begin{raggedright}
\quad{}\quad{}for each $x\in E$,
\par\end{raggedright}

\begin{raggedright}
\quad{}\quad{}\quad{}\quad{}for each $x^{\prime}\in E$, set
\[
M_{n}^{y}(x,x^{\prime})=\begin{cases}
\dfrac{M_{n}(x,x^{\prime})g_{n}(x^{\prime},y_{n})\phi_{n+1}(x^{\prime},y_{n+1:0})}{\phi_{n}(x,y_{n:0})}, & \;\phi_{n}(x,y_{n:0})>0,\\
M_{n}(x,x^{\prime}), & \;\phi_{n}(x,y_{n:0})=0,
\end{cases}
\]

\par\end{raggedright}

\begin{raggedright}
\quad{}\quad{}\quad{}\quad{}sample $\xi_{n}^{x}\sim M_{n}^{y}(x,\cdot)$
and set $\Phi_{n}(x)=\xi_{n}^{x}$
\par\end{raggedright}

\begin{raggedright}
\quad{}\quad{}set $\Phi_{n-1,0}=\Phi_{n,0}\circ\Phi_{n}$
\par\end{raggedright}

\begin{raggedright}
\quad{}\quad{}set $n=n-1$
\par\end{raggedright}

\begin{raggedright}
\quad{}\quad{}for each $x\in E$, set
\[
\phi_{n}(x,y_{n:0})=\sum_{x^{\prime}\in E}M_{n}(x,x^{\prime})g_{n}(x^{\prime},y_{n})\phi_{n+1}(x^{\prime},y_{n+1:0}).
\]

\par\end{raggedright}

\begin{raggedright}
return $\Phi_{n,0}(x)$, for any $x\in E$.
\par\end{raggedright}

\protect\caption{\label{alg:Perfect-sampling-for-hmm}Perfect sampling for the hidden
Markov model}

\end{algorithm}

\subsection{Successful coupling and conditional ergodicity}

With a little further technical work, we can relate the successful
coupling in the sense of (\ref{eq:P_T_n_all_finite_hmm}) to the conditional
ergodicity properties of the HMM. Our next step is to perform some
careful accounting of certain $\sigma$-algebras to help us transfer
results backwards and forwards between the measurable space $(\Omega,\mathcal{F})$
underlying the HMM and the ``marginal'' space $(\Omega^{X},\mathcal{F}^{\tilde{X}})$;
the attentive reader will have noticed that under the definitions
of Section~\ref{sub:The-model-hmm}, $\mathcal{F}_{I}^{X}$ consists
of subsets of $\Omega$, where as $\mathcal{F}_{I}^{\tilde{X}}$ consists
of subsets $\Omega^{X}$. On the other hand, $\mathcal{F}_{I}^{\tilde{X}}$
coincides with the object in Section~\ref{sec:Tail-Triviality-and}
denoted there by $\mathcal{F}_{I}^{X}$, and in terms of which Theorem~\ref{thm:tail}
is phrased. The resolution of this issue is provided by the following
technical lemma, whose proof is given in the appendix.
\begin{lem}
\label{lem:F_X_F_X_tilde}With the definitions of Section~\ref{sub:The-model-hmm}
in force,
\[
\mathcal{F}_{I}^{X}=\{A\times\Omega^{Y};A\in\mathcal{F}_{I}^{\tilde{X}}\}\quad\forall I\subset\mathbb{T},
\]
and,
\begin{equation}
\bigcap_{n\in\mathbb{T}}\mathcal{F}_{]-n,0]}^{X}=\left\{ A\times\Omega^{Y};A\in\bigcap_{n\in\mathbb{T}}\mathcal{F}_{]-n,0]}^{\tilde{X}}\right\} .\label{eq:tail_x_tail_x_tilde}
\end{equation}

\end{lem}
Lemma~\ref{lem:F_X_F_X_tilde} allows us to set up correspondence
between probabilities on $\mathcal{F}^{\tilde{X}}$ and $\mathcal{F}^{X}$,
and in particular we have:
\begin{lem}
\label{lem:P_bold_P} There exists a probability kernel $P^{\cdot}(\cdot):\Omega^{Y}\times\mathcal{F}^{\tilde{X}}\rightarrow[0,1]$
and a set $\tilde{H}\in\mathcal{F}^{\tilde{Y}}$ of $\mathbf{P}\circ Y^{-1}$
probability $1$, such that for all $y\in\tilde{H}$
\[
P^{y}(\{\tilde{X}_{n}=x_{n},\ldots,\tilde{X}_{0}=x_{0}\})=\mu_{n}^{y}(x_{n})\prod_{k=n+1}^{0}M_{k}^{y}(x_{k-1},x_{k}).
\]
The function $\mathbf{P}^{\mathcal{F}^{Y}}:\Omega\times\mathcal{F}^{X}\rightarrow[0,1]$
defined by
\[
\mathbf{P}^{\mathcal{F}^{Y}}(\omega,A\times\Omega^{Y}):=P^{Y(\omega)}(A),\quad A\in\mathcal{F}^{\tilde{X}},
\]
 is a probability kernel, and for each $A\in\mathcal{F}^{\tilde{X}}$,
\begin{equation}
\mathbf{P}^{\mathcal{F}^{Y}}(\omega,A\times\Omega^{Y})=\mathbf{P}(A\times\Omega^{Y}|\mathcal{F}^{Y})(\omega),\quad\text{ for }\mathbf{P}\text{-almost all }\mbox{\ensuremath{\omega\in\Omega}}.\label{eq:P_f_is_a_version}
\end{equation}

\end{lem}
The proof is in the appendix. For any $y\in\Omega^{Y}$, we denote
by $\Pi_{M^{y}}$ the set of all sequences of absolute probabilities
for $M^{y}=(M_{n}^{y})_{n\in\mathbb{T}}$. The set $\Pi_{M^{y}}$
is non-empty by Fact~\ref{fact:there-always-exists}. For any $y\in\Omega^{Y}$
we shall write generically $\pi^{y}$ for a member of $\Pi_{M^{y}}$
(we do not claim measurable dependence of $\pi_{n}^{y}$ on $y$ except
at least in the case of $\pi_{n}^{y}=\mu_{n}^{y}$ with the latter
as in Lemma~\ref{lem:existence_of_pi^y}), and, by arguments only
superficially different (we omit the details) to those used in the
proof of Lemma~\ref{lem:P_bold_P}, for any such $\pi^{y}\in\Pi_{M^{y}}$
there exists a probability measure $P_{\pi^{y}}$ on $\mathcal{F}^{\tilde{X}}$
such that
\begin{equation}
P_{\pi^{y}}(\{\tilde{X}_{n}=x_{n},\ldots,\tilde{X}_{0}=x_{0}\})=\pi_{n}^{y}(x_{n})\prod_{k=n+1}^{0}M_{k}^{y}(x_{k-1},x_{k}),\label{eq:P_sub_pi}
\end{equation}
and
\[
\mathbf{P}_{\pi^{y}}(A\times\Omega^{Y}):=P_{\pi^{y}}(A),\quad A\in\mathcal{F}^{\tilde{X}},
\]
 defines a probability measure on $\mathcal{F}^{X}$. We now have
the technical and notational devices to state and prove the following
theorem, which characterizes almost sure success of the coupling.
\begin{thm}
\label{thm:cond_ergo_perf_sampling}The following are equivalent.\vspace{0.15cm}

\noindent 1. $\mathbf{Q}^{Y(\omega)}(\bigcap_{n\in\mathbb{T}}\{T_{n}>-\infty\})=1$
for $\mathbf{P}$-almost all $\omega$,

\noindent 2. There exists a set $H\in\mathcal{F}$ such that $\mathbf{P}(H)=1$
and
\[
\mathbf{P}^{\mathcal{F}^{Y}}(\omega,A)=\mathbf{P}^{\mathcal{F}^{Y}}(\omega,A)^{2}=\mathbf{P}_{\pi^{Y(\omega)}}(A),
\]
 \quad{}\ for all $\omega\in H$, $A\in\bigcap_{n\in\mathbb{T}}\mathcal{F}_{]-\infty,n]}^{X}$
and $\pi^{Y(\omega)}\in\Pi_{M^{Y(\omega)}}$.\end{thm}
\begin{proof}
When 1. holds, there exists $H\in\mathcal{F}$ with $\mathbf{P}(H)=1$
such that for all $\omega\in H$ the following hold: for all $n\in\mathbb{T}$,
$\mathbf{Q}^{Y(\omega)}(T_{n}>-\infty)=1$; then via Proposition~\ref{prop:fixed_y_weak_ergo_hmm}
and Proposition~\ref{prop:weak_ergo_charac}, $\text{card}(\Pi_{M^{Y(\omega)}})=1$;
then by an application of Theorem~\ref{thm:tail} with the $\mathbf{P}_{\pi}$
appearing there taken to be $P^{Y(\omega)}(\cdot)$, and Lemma~\ref{lem:F_X_F_X_tilde},
we have $P^{Y(\omega)}(\tilde{A})=P^{Y(\omega)}(\tilde{A})^{2}=P_{\pi^{Y(\omega)}}(\tilde{A})$
for all $\tilde{A}\in\bigcap_{n\in\mathbb{T}}\mathcal{F}_{]-\infty,n]}^{\tilde{X}}$
and $\pi^{Y(\omega)}\in\Pi_{M^{Y(\omega)}}$. Lemmata~\ref{lem:F_X_F_X_tilde}
and~\ref{lem:P_bold_P} then give $\mathbf{P}^{\mathcal{F}^{Y}}(\omega,A)=\mathbf{P}^{\mathcal{F}^{Y}}(\omega,A)^{2}=\mathbf{P}_{\pi^{Y(\omega)}}(A)$
for all $A\in\bigcap_{n\in\mathbb{T}}\mathcal{F}_{]-\infty,n]}^{X}$,
which establishes 2.

If 2. holds, we apply this chain of reasoning in reverse to establish
1. The details are omitted in order to avoid repetition.
\end{proof}

\section{Discussion\label{sec:Discussion}}

Throughout Section~\ref{sec:Discussion} the definitions and constructions
of Section~\ref{sec:HMM} are in force. In particular it is timely
to recall that the law of the HMM, $\mathbf{P}$, has the defining
ingredients:
\begin{itemize}
\item $M=(M_{n})_{n\in\mathbb{T}}$ a sequence of Markov kernels on $E$
\item $\pi=(\pi)_{n\in\mathbb{T}}\in\Pi_{M}$ a sequence of absolute probabilities
for $M$
\item $G=(G_{n})_{n\in\mathbb{T}}$ a sequence of probability kernels, each
acting from $E$ to $F$ and such that for each $n$, $G_{n}(x,dy)=g_{n}(x,y)\psi(dy)$
\end{itemize}
We shall consider various combinations of the following assumptions.
\begin{assumption}
\label{assu:signal_ergo}Under $\mathbf{P}$ the signal process is
ergodic, in that $\mathbf{P}(A)=\mathbf{P}(A)^{2}$ for all $A\in\bigcap_{n\in\mathbb{T}}\mathcal{F}_{]-\infty,n]}^{X}$
\end{assumption}

\begin{assumption}
\label{assu:non_degeneracy}The observations are non-degenerate, in
that $g_{n}(x,y)>0$ for all $n\in\mathbb{T}$, $x\in E$ and $y\in F$.
\end{assumption}

\begin{assumption}
\label{assu:time_homog}The signal transitions, absolute probabilities
and observation kernels do not depend on time, in that $M_{n}=M_{0}$,
$\pi_{n}=\pi_{0}$ and $G_{n}=G_{0}$ for all $n\in\mathbb{T}$.
\end{assumption}
\noindent Let us briefly comment on these assumptions.

Assumption~\ref{assu:signal_ergo} does not imply that the backward
products of $M$ are weakly ergodic: the latter is, by Proposition~\ref{prop:weak_ergo_charac}
and Theorem~\ref{thm:tail}, equivalent to the \emph{simultaneous}
tail triviality of $X$ under the probability measures over paths
in $\Omega^{X}$ derived from \emph{all} members of $\Pi_{M}$, whereas
Assumption~\ref{assu:signal_ergo} involves only the particular $\pi\in\Pi_{M}$
used to construct $\mathbf{P}$.

Assumption~\ref{assu:non_degeneracy} is the same type of assumption
employed by \citet{van2009stability} and ensures that information
from the observations cannot rule out with certainty any particular
hidden state. We shall use several times the fact that when this assumption
holds, $\phi_{n}(x,y_{n:0})>0$ for all $n$, $x$ and $y_{n},\ldots,y_{0}$,
which is established by a simple induction.

Assumption~\ref{assu:time_homog} sacrifices some of the generality
of the HMM, but serves to simplify our discussions. Note that when
this assumption holds the signal process $X$ is stationary under
$\mathbf{P}$ with $\mathbf{P}(X_{n}\in\cdot)=\pi_{0}(\cdot)$ for
all $n\in\mathbb{T}$.

\subsection{The connection to filter stability\label{sub:The-connection-to}}

Throughout Section~\ref{sub:The-connection-to} we adopt Assumption~\ref{assu:time_homog}.
Let $(X^{+},Y^{+})$ be the time-reversal of $(X,Y)$, i.e. $X_{n}^{+}=X_{-n}$,
$Y_{n}^{+}=Y_{-n}$, $n\in\mathbb{N}$. For some probability distribution
$\bar{\pi}_{0}$, not necessarily an invariant distribution for $M_{0}$,
but such that $\bar{\pi}_{0}\ll\pi_{0}$, let $\mathbf{\bar{P}}$
be the probability measure on $(\Omega,\mathcal{F})$ under which
$(X^{+},Y^{+})$ has the same transition probabilities as under $\mathbf{P}$
but $X_{0}^{+}\sim\bar{\pi}_{0}$, so, since $X_{0}^{+}=X_{0}$,

\[
\frac{d\mathbf{\bar{P}}}{d\mathbf{P}}(X,Y)=\frac{d\bar{\pi}_{0}}{d\pi_{0}}(X_{0}),\quad\mathbf{P}-a.s.
\]
For $n\in\mathbb{T}$ let $\rho_{n}^{Y}$ and $\bar{\rho}_{n}^{Y}$
be respectively regular conditional probabilities of the form $\mathbf{P}(X_{n}\in\cdot|\mathcal{F}_{[n,0]}^{Y})$
and $\bar{\mathbf{P}}(X_{n}\in\cdot|\mathcal{F}_{[n,0]}^{Y})$, so
$\rho_{n}^{Y}$ (resp. $\bar{\rho}_{n}^{Y}$) is a filtering distribution
under $\mathbf{P}$ (resp. $\bar{\mathbf{P}}$) for the time-reversed
HMM $(X^{+},Y^{+})$. Amongst various notions of forgetting associated
with HMM's, \emph{asymptotic filter stability} (in mean) is the phenomenon:
\begin{equation}
\lim_{n\rightarrow-\infty}\bar{\mathbf{E}}[\|\rho_{n}^{Y}-\bar{\rho}_{n}^{Y}\|]=0.\label{eq:fitler_stability}
\end{equation}
As discussed in \citep{van2009stability,intrinsic_filter_chapter}
and references therein, it is now well known that ergodicity of the
signal as per Assumption~\ref{assu:signal_ergo} is, alone, not enough
to establish filter stability (in various senses), see Section~\ref{sub:counter_examples_deg_obs}
for a counter-example. However (\ref{eq:fitler_stability}) does hold
if
\begin{equation}
\bigcap_{n\in\mathbb{T}}\mathcal{F}_{]-\infty,0]}^{Y}\vee\mathcal{F}_{]-\infty,n]}^{X}=\mathcal{F}_{]-\infty,0]}^{Y},\quad\mathbf{P}-\text{a.s.},\label{eq:filter_stab_sigma_1}
\end{equation}
see \citep{intrinsic_filter_chapter} for a proof, and using a result
of \citet{weizsacker1983exchanging}, when Assumption~\ref{assu:signal_ergo}
holds a necessary and sufficient condition for (\ref{eq:filter_stab_sigma_1})
is:
\begin{equation}
\bigcap_{n\in\mathbb{T}}\mathcal{F}_{]-\infty,n]}^{X}\;\text{ is }\;\mathbf{P}^{\mathcal{F}^{Y}}(\omega,\cdot)\text{-a.s. trivial, for }\mathbf{P}\text{-a.e. }\omega.\label{eq:cond_ergodicity}
\end{equation}
Here, $\mathbf{P}^{\mathcal{F}^{Y}}$is the object defined in Lemma~\ref{lem:P_bold_P}
and appearing in Theorem~\ref{thm:cond_ergo_perf_sampling}: it is
a version of $\mathbf{P}(\cdot|\mathcal{F}^{Y})$ as a regular conditional
distribution over $\mathcal{F}^{X}$ given $\mathcal{F}^{Y}$. Thus,
we see that if the coupling for the HMM is a.s. successful, in the
sense that condition 1. of Theorem~\ref{thm:cond_ergo_perf_sampling}
holds, then condition 2. of that Theorem holds, implying (\ref{eq:cond_ergodicity}),
and therefore (\ref{eq:fitler_stability}). Thus asymptotic filter
stability for the reversed HMM is a necessary condition for a.s. successful
coupling. However as we shall discuss in Section~\ref{sub:counter_examples_red},
the condition (\ref{eq:cond_ergodicity}) is in general weaker than
the \emph{simultaneous} tail triviality in condition 2. of Theorem~\ref{thm:cond_ergo_perf_sampling}.

\subsection{Counter examples to successful coupling}

\subsubsection{Degenerate observations\label{sub:counter_examples_deg_obs}}

The purpose of this section is to show how consideration of unicity
of absolute probabilities and a.s. success of the coupling bring a
fresh perspective on a well known counter-example to filter stability
due to \citet{baxendale2004asymptotic}. Out starting point is to
observe that for any $y\in\tilde{H}$ where $\tilde{H}$ is as in
Lemma~\ref{lem:P_bold_P}, the tower property of conditional expectation
and the Markov property of $\tilde{X}$ under $P^{y}$ give
\begin{eqnarray}
 &  & P^{y}\left(\left.\tilde{X}_{k}=x\right|\bigcap_{n\in\mathbb{T}}\mathcal{F}_{]-\infty,n]}^{\tilde{X}}\right)\nonumber \\
 &  & =E^{y}\left[\left.P^{y}\left(\left.\tilde{X}_{k}=x\right|\sigma(\tilde{X}_{k-1})\vee\bigcap_{n\in\mathbb{T}}\mathcal{F}_{]-\infty,n]}^{\tilde{X}}\right)\right|\bigcap_{n\in\mathbb{T}}\mathcal{F}_{]-\infty,n]}^{\tilde{X}}\right]\nonumber \\
 &  & =E^{y}\left[\left.M_{k}^{y}(\tilde{X}_{k-1},x)\right|\bigcap_{n\in\mathbb{T}}\mathcal{F}_{]-\infty,n]}^{\tilde{X}}\right]\nonumber \\
 &  & =\sum_{x^{\prime}}P^{y}\left(\left.\tilde{X}_{k-1}=x^{\prime}\right|\bigcap_{n\in\mathbb{T}}\mathcal{F}_{]-\infty,n]}^{\tilde{X}}\right)M_{k}^{y}(x^{\prime},x),\quad P^{y}-a.s.\label{eq:tail_triv_abs_probs}
\end{eqnarray}
where $E^{y}$ denotes expectation w.r.t. $P^{y}$. So, if $\bigcap_{n\in\mathbb{T}}\mathcal{F}_{]-\infty,n]}^{\tilde{X}}$
is \emph{not} $P^{y}$-a.s. trivial (cf. (\ref{eq:cond_ergodicity})
via Lemmata~\ref{lem:F_X_F_X_tilde} and~\ref{lem:P_bold_P}), regular
conditional probabilities of the form $P^{y}(\tilde{X}_{n}=x|\bigcap_{n\in\mathbb{T}}\mathcal{F}_{]-\infty,n]}^{\tilde{X}})$
give rise to absolute probabilities for $M^{y}$ distinct from $\mu^{y}$,
hence $\text{card}(\Pi_{M^{y}})>1$. The following is a concrete example
of this phenomenon. Throughout the remainder of Section~\ref{sub:counter_examples_deg_obs}
Assumption~\ref{assu:time_homog} is in force.

Let $E=\{0,1,2,3\}$ and
\[
M_{0}(x,x)=1/2,\quad M_{0}(x+1\;\text{mod}\,4,x)=1/2.
\]
$M_{0}$ obviously has a unique invariant distribution and Assumption~\ref{assu:signal_ergo}
holds. Let $Y_{n}=\mathbf{1}_{\{X_{n}\in\{1,3\}\}}$. Then the time-reversed
model $(X^{+},Y^{+})$ coincides up to a relabeling of states with
\citep[Example 1.1]{intrinsic_filter_chapter}, and as explained therein
(\ref{eq:filter_stab_sigma_1}) does not hold, so neither does (\ref{eq:cond_ergodicity}).
It follows by Theorem~\ref{thm:cond_ergo_perf_sampling} that the
coupling is not a.s. successful. We can also verify that for this
model $\text{card}(\Pi_{M^{y}})>1$ for any $y\in\Omega^{Y}$ by direct
calculation in connection with (\ref{eq:tail_triv_abs_probs}), so
that the lack of successful coupling can also be deduced from Proposition~\ref{prop:fixed_y_weak_ergo_hmm}.

Fix any $y=(y_{n})_{n\in\mathbb{T}}\in\Omega^{Y}$. A simple induction
argument provides that for any $n\in\mathbb{T}$ and $x\in E$,
\[
\phi_{n}(x,y_{n:0})=2^{n-1},
\]
and therefore $M_{n}^{y}$ is given by
\begin{equation}
M_{n}^{y}(x,x')=\begin{cases}
2M_{n}(x,x'),\quad & y_{n}=x'\text{ mod}\,2,\\
0, & \text{otherwise}.
\end{cases}\label{eq:M^y_deg}
\end{equation}
From this we observe that for each $x\in E$ either $M_{n}^{y}(x,x)=1$
or $M_{n}^{y}(x,x-1\:\mod\,4)=1$, and in matrix form $M_{n}^{y}$
is as follows, with the case $y_{n}=0$ on the left and the case $y_{n}=1$
on the right:
\[
\left[\begin{array}{cccc}
1 & 0 & 0 & 0\\
1 & 0 & 0 & 0\\
0 & 0 & 1 & 0\\
0 & 0 & 1 & 0
\end{array}\right],\quad\left[\begin{array}{cccc}
0 & 0 & 0 & 1\\
0 & 1 & 0 & 0\\
0 & 1 & 0 & 0\\
0 & 0 & 0 & 1
\end{array}\right].
\]
 Let $A_{n}^{y}:=\{x:y_{n}=x\text{ mod}\,2\}$. By (\ref{eq:M^y_deg}),
$x^{\prime}\notin A_{n}^{y}\Rightarrow M_{n}^{y}(x,x')=0$, so we
observe that any $\pi^{y}\in\Pi_{M^{y}}$, i.e. satisfying the equations
\[
\pi_{n}^{y}(x)=\sum_{x'\in E}\pi_{n-1}^{y}(x^{\prime})M_{n}^{y}(x^{\prime},x),\quad\forall x\in E,n\in\mathbb{T},
\]
must be such that for all $n\in\mathbb{T}$,
\begin{equation}
\pi_{n}^{y}(x)=\begin{cases}
\sum_{x^{\prime}\in A_{n-1}^{y}}\pi_{n-1}^{y}(x^{\prime})M_{n}^{y}(x^{\prime},x),\quad & x\in A_{n}^{y},\\
0, & x\notin A_{n}^{y}.
\end{cases}\label{eq:pi_n_recurse}
\end{equation}
Via some simple manipulations it then follows that if the values $\{\pi_{n}^{y}(x)\}_{x\in A_{n}^{y}}$
are fixed, (\ref{eq:pi_n_recurse}) provides two equations which can
be solved for the two values $\{\pi_{n-1}^{y}(x)\}_{x\in A_{n-1}^{y}}$.
So, if for any $w\in[0,1]$ we set $\pi_{0}^{y}(x):=w\mathbb{I}[x=y_{0}]+(1-w)\mathbb{I}[x=y_{0}+2]$
and then recursively solve (\ref{eq:pi_n_recurse}) for $(\pi_{n}^{y})_{n\in\mathbb{T}\setminus\{0\}}$,
we obtain by construction a sequence $(\pi_{n}^{y})_{n\in\mathbb{T}}\in\Pi_{M^{y}}$
uniquely defined by $w$. Distinct values of $w$ thus giving rise
to distinct members of $\Pi_{M^{y}}$, we therefore have $\text{card}(\Pi_{M^{y}})>1$.
Moreover, when $w=0$ or $w=1$ the measure $P_{\pi^{y}}$ defined
as in (\ref{eq:P_sub_pi}) fixes all its mass on a single point in
$\Omega^{X}$, and $\pi_{n}^{y}$ is a version of $P^{y}(\tilde{X}_{n}\in\cdot\,|\bigcap_{n\in\mathbb{T}}\mathcal{F}_{]-\infty,n]}^{\tilde{X}})$
as in (\ref{eq:tail_triv_abs_probs}).

\subsubsection{Reducible signal\label{sub:counter_examples_red}}

Throughout Section~\ref{sub:counter_examples_red} Assumptions~\ref{assu:non_degeneracy}
and~\ref{assu:time_homog} are in force. The purpose of this example
is to illustrate that condition 2. of Theorem~\ref{thm:cond_ergo_perf_sampling}
is strictly stronger than (\ref{eq:cond_ergodicity}), and so the
coupling may fail to be a.s. successful even when (\ref{eq:cond_ergodicity})
holds. The main point here is that reducibility of $M_{0}$ is not
ruled out by (\ref{eq:cond_ergodicity}), and may compromise weak
ergodicity of the backward products of $M^{y}$. Indeed, Let $E=\{0,1,2,3\}$
and let $M_{0}$ be given by the matrix:
\[
\left[\begin{array}{cccc}
\frac{1}{2} & \frac{1}{2} & 0 & 0\\
\frac{1}{2} & \frac{1}{2} & 0 & 0\\
0 & 0 & \frac{1}{2} & \frac{1}{2}\\
0 & 0 & \frac{1}{2} & \frac{1}{2}
\end{array}\right]
\]
Obviously $\pi_{0}:=\begin{bmatrix}1/2 & 1/2 & 0 & 0\end{bmatrix}$
is an invariant distribution for $M_{0}$ and with this choice of
$\pi_{0}$ Assumption~\ref{assu:signal_ergo} is satisfied, since
under the probability measure $\mathbf{P}$, which is constructed
using $M_{0}$ and $\pi_{0}$, the $(X_{n})_{n\in\mathbb{T}}$ are
then i.i.d. according to $\pi_{0}$.

With $M_{0}$ as given by the above matrix, $\phi_{n}(0,y_{n:0})=\phi_{n}(1,y_{n:0})$
and $\phi_{n}(2,y_{n:0})=\phi_{n}(3,y_{n:0})$ for all $n$ and $y_{n},\ldots,y_{0}$,
and $M_{n}^{y}$ is given by the matrix:
\[
\left[\begin{array}{cccc}
\frac{g_{0}(0,y_{n})}{g_{0}(0,y_{n})+g_{0}(1,y_{n})} & \frac{g_{0}(1,y_{n})}{g_{0}(0,y_{n})+g_{0}(1,y_{n})} & 0 & 0\\
\frac{g_{0}(0,y_{n})}{g_{0}(0,y_{n})+g_{0}(1,y_{n})} & \frac{g_{0}(1,y_{n})}{g_{0}(0,y_{n})+g_{0}(1,y_{n})} & 0 & 0\\
0 & 0 & \frac{g_{0}(2,y_{n})}{g_{0}(2,y_{n})+g_{0}(3,y_{n})} & \frac{g_{0}(3,y_{n})}{g_{0}(2,y_{n})+g_{0}(3,y_{n})}\\
0 & 0 & \frac{g_{0}(2,y_{n})}{g_{0}(2,y_{n})+g_{0}(3,y_{n})} & \frac{g_{0}(3,y_{n})}{g_{0}(2,y_{n})+g_{0}(3,y_{n})}
\end{array}\right]
\]

The equalities in the statement of Lemma~\ref{lem:existence_of_pi^y}
are satisfied if we take
\[
\mu_{n}^{y}(x)=\frac{\pi_{0}(x)g_{0}(x,y_{n})}{\sum_{z\in E}\pi_{0}(z)g_{0}(z,y_{n})}=\frac{g_{0}(x,y_{n})}{g_{0}(0,y_{n})+g_{0}(1,y_{n})}\mathbb{I}[x\in\{0,1\}],
\]
and then $(X_{n})_{n\in\mathbb{T}}$ are independent under $\mathbf{P}^{\mathcal{F}^{Y}}(\omega,\cdot)$
for any $\omega\in\Omega$, so that (\ref{eq:cond_ergodicity}) holds
by the Kolmogorov 0-1 law. However, condition 2. of Theorem~\ref{thm:cond_ergo_perf_sampling}
does not hold: to see this first note that for any $y\in\Omega^{Y}$
, $\pi^{y}=(\pi_{n}^{y})_{n\in\mathbb{T}}$ with $\pi_{n}^{y}(x)\propto\mathbb{I}[x\in\{2,3\}]g_{0}(x,y_{n})$
defines a sequence of absolute probabilities for $M^{y}$, distinct
from $\mu^{y}$. Then, with $A:=\{\omega:X_{n}(\omega)\in\{0,1\}\; i.o.\}\in\bigcap_{n}\mathcal{F}_{]-\infty,n]}^{X}$
,
\[
0=\mathbf{P}_{\pi^{Y(\omega)}}(A)\neq\mathbf{P}^{\mathcal{F}^{Y}}(\omega,A)=1,\quad\forall\omega\in\Omega.
\]
Thus we conclude that condition 1 of Theorem~\ref{thm:cond_ergo_perf_sampling}
does not hold, i.e., the coupling is not a.s. successful. Of course,
this could have been verified more directly using Proposition~\ref{prop:fixed_y_weak_ergo_hmm};
the backward products of $M^{y}$ are clearly not weakly ergodic and
in fact $\mathbf{Q}^{y}(\bigcap_{n\in\mathbb{T}}\{T_{n}=-\infty\})=1$.

\subsection{Verifiable conditions for successful coupling}

Our next aim is to present some sufficient conditions for condition
1. of Theorem~\ref{thm:cond_ergo_perf_sampling} to hold. We shall
make several uses of the following lemma, the proof of which is mostly
technical and is given in the appendix.
\begin{lem}
\label{lem:beta_My_bounds}Suppose that for some $n<0$ there exists
$k\in\{n+1,\ldots,0\}$, a probability distribution $\nu$ and constants
$(\epsilon^{-},\epsilon^{+})\in]0,\infty[$ such that
\begin{equation}
\epsilon^{-}\nu(x^{\prime})\leq M_{n,k}(x,x^{\prime})\leq\epsilon^{+}\nu(x^{\prime}),\quad\forall(x,x^{\prime})\in E^{2}.\label{eq:M_n,k_mixing}
\end{equation}
Then the following hold:

\noindent 1. If $k=n+1$ and $\phi_{n+1}(x,y_{n+1:0})>0$ for all
$x$ and $y_{n+1},\ldots,y_{0}$, then
\begin{equation}
\sup_{y\in\Omega^{Y}}\beta(M_{n,k}^{y})\leq1-\frac{\epsilon^{-}}{\epsilon^{+}}<1.\label{eq:M^y_dobrushin_bound1}
\end{equation}
2. If $k>n+1$,
\begin{equation}
\phi_{k+1}(x,y_{k+1:0})>0,\quad\forall(x,y_{k+1},\ldots,y_{0})\in E\times F^{k}\label{eq:phi_k_positive}
\end{equation}
and $g_{j}(x,y)>0$ for all $x,y$ and $j=n+1,\ldots,k$, then
\begin{equation}
\beta(M_{n,k}^{y})\leq1-\frac{\epsilon^{-}}{\epsilon^{+}}\prod_{j=n+1}^{k}\frac{g_{j}^{-}(y_{j})}{g_{j}^{+}(y_{j})}<1,\quad\forall y=(y_{n})_{n\in\mathbb{T}}\in\Omega^{Y},\label{eq:M^y_dobrushin_bound_2}
\end{equation}
where $g_{j}^{-}(y):=\min_{x}g_{j}(x,y)$, $g_{j}^{+}(y):=\max_{x}g_{j}(x,y)$.
\end{lem}

\subsubsection{Almost surely successful coupling\label{sub:Almost-sure-successful}}

Throughout Section~\ref{sub:Almost-sure-successful} Assumption~\ref{assu:time_homog}
is in force. In the examples of Sections~\ref{sub:counter_examples_deg_obs}
and~\ref{sub:counter_examples_red} it is respectively the issues
of degeneracy of the observations and reducibility of $M_{0}$ which
caused problems for successful coupling. Our next aim is to illustrate
that once these two issues are ruled out, condition 1. of Theorem~\ref{thm:cond_ergo_perf_sampling}
holds.

Suppose that $\pi_{0}$ is the unique invariant distribution of $M_{0}$,
\begin{equation}
\pi_{0}(x)>0,\;\forall x\in E,\quad\text{and}\quad\lim_{n\rightarrow\infty}M_{0}^{(n)}(x,x^{\prime})-\pi_{0}(x^{\prime})=0,\quad\quad\forall(x,x^{\prime})\in E^{2}.\label{eq:M_0_weak_ergo_pi_+ve}
\end{equation}
It follows that there exists a probability distribution $\nu$, $(\epsilon^{-},\epsilon^{+})\in]0,\infty[$
and $m\geq1$ such that
\begin{equation}
\epsilon^{-}\nu(x^{\prime})\leq M_{0}^{(m)}(x,x^{\prime})\leq\epsilon^{+}\nu(x^{\prime}),\quad\forall(x,x^{\prime})\in E^{2}.\label{eq:M_0_m_step}
\end{equation}

We shall now argue that, if we adopt also Assumption~\ref{assu:non_degeneracy},
then for each $k\in\mathbb{T}$,
\begin{equation}
\lim_{n\rightarrow-\infty}\beta(M_{n,k}^{Y})=0,\quad\mathbf{P}-a.s.,\label{eq:beta_as_to_zero}
\end{equation}
which is, via Propositions~\ref{prop:fixed_y_weak_ergo_hmm} and~\ref{prop:weak_ergo_charac},
equivalent to condition 1. of Theorem~\ref{thm:cond_ergo_perf_sampling}.

Using the submultiplicativity of the Dobrushin coefficient and part
2. of Lemma~\ref{lem:beta_My_bounds}, we have for any $y=(y_{n})_{n\in\mathbb{T}}$,
and $n<k\in\mathbb{T}$,
\begin{eqnarray*}
\beta(M_{n,k}^{y}) & \leq & \prod_{i=0}^{\left\lfloor \frac{k-n}{m}\right\rfloor -1}\beta(M_{k-(i+1)m,k-im}^{y})\\
 & \leq & \prod_{i=0}^{\left\lfloor \frac{k-n}{m}\right\rfloor -1}(1-f(y_{k-(i+1)m+1},\ldots,y_{k-im})),
\end{eqnarray*}
where $f:F^{m}\rightarrow]0,1]$ is given b
\[
f(y_{1},\ldots,y_{m}):=\frac{\epsilon^{-}}{\epsilon^{+}}\prod_{j=1}^{m}\frac{g_{0}^{-}(y_{j})}{g_{0}^{+}(y_{j})},\quad(y_{1},\ldots,y_{m})\in F^{m}.
\]
Since for any sequence $(a_{i})_{i\in\mathbb{N}}$ with values in
$]0,1]$, $\prod_{i=0}^{\infty}(1-a_{i})=0\Leftrightarrow\sum_{i=0}^{\infty}a_{i}=\infty$,
in order to establish (\ref{eq:beta_as_to_zero}) it suffices to show

\begin{equation}
\sum_{i=0}^{\infty}f(Y_{k-(i+1)m+1},\ldots,Y_{k-im})=\infty,\quad\mathbf{P}-a.s.\label{eq:sum_f_blow_up}
\end{equation}
To this end let $Z_{i}^{(k)}=(X_{k-(i+1)m+1},\ldots,X_{k-im},Y_{k-(i+1)m+1},\ldots,Y_{k-im})$.
The time reversed bivariate process $(X^{+},Y^{+})$ is a stationary
Markov chain under $\mathbf{P}$, and it follows from (\ref{eq:M_0_weak_ergo_pi_+ve})
and the conditional independence structure of the HMM that the transition
kernel of $(X^{+},Y^{+})$ has a unique invariant distribution $\pi_{0}(dx)g_{0}(x,y)\psi(dy)$,
and is uniformly ergodic, in the sense of \citet[Chapter 16]{meyn2009markov}.
Some simple but tedious calculations show that under $\mathbf{P}$,
$(Z_{i}^{(k)})_{i\in\mathbb{N}}$ is then also a stationary Markov
chain, with transition kernel which admits a unique invariant distribution
and which is uniformly ergodic. So by the strong law of large numbers
for stationary and ergodic Markov chains,
\begin{equation}
\lim_{n\rightarrow\infty}\frac{1}{n}\sum_{i=0}^{n-1}f(Y_{k-(i+1)m+1},\ldots,Y_{k-im})=\mathbb{E}[f(Y_{-m+1},\ldots,Y_{0})],\quad\mathbf{P}-a.s.\label{eq:mc_SLLN}
\end{equation}
Since $f$ is strictly positive the expectation in (\ref{eq:mc_SLLN})
is strictly positive, hence (\ref{eq:sum_f_blow_up}) holds, hence
(\ref{eq:beta_as_to_zero}) holds.

\subsubsection{Surely successful coupling}

In practice one is typically presented with an observation sequence
which is not necessarily distributed according to $\mathbf{P}$. It
may then be of some concern that even if one (and then both) of the
conditions of Theorem~\ref{thm:cond_ergo_perf_sampling} holds, the
coupling may fail to be successful for $y$ in a set of observation
sequences which has zero probability under $\mathbf{P}$. In this
section, we discuss some simple sufficient conditions for the stronger
requirement that
\begin{equation}
\mathbf{Q}^{y}\left(\bigcap_{n\in\mathbb{T}}\left\{ T_{n}>-\infty\right\} \right)=1,\quad\forall y\in\Omega^{Y}.\label{eq:sure_termination}
\end{equation}

Suppose that
\begin{equation}
\inf_{n\in\mathbb{T}}\min_{(x,x^{\prime})\in E^{2}}M_{n}(x,x^{\prime})>0.\label{eq:M_n_unifom_positive}
\end{equation}
and assume that for all $y=(y_{n})_{n\in\mathbb{T}}\in\Omega^{Y}$
\begin{equation}
\forall n\in\mathbb{T},\;\exists x:g_{n}(x,y_{n})>0.\label{eq:g_positive_somewhere}
\end{equation}
An application of part 1. of Lemma~\ref{lem:beta_My_bounds} gives
\[
\sup_{y\in\Omega^{Y}}\sup_{n\in\mathbb{N}}\beta(M_{n}^{y})<1,
\]
so for any $k\in\mathbb{T}$ and $y\in\Omega^{Y}$,
\[
\lim_{n\rightarrow-\infty}\beta(M_{n,k}^{y})\leq\lim_{n\rightarrow-\infty}\prod_{n}^{k}\beta(M_{j}^{y})=0,
\]
which via Proposition~\ref{prop:fixed_y_weak_ergo_hmm} and Proposition~\ref{prop:weak_ergo_charac}
gives (\ref{eq:sure_termination}).

The reader can easily verify that Part 2. of Lemma~\ref{lem:beta_My_bounds}
can be used to show that (\ref{eq:sure_termination}) holds under
conditions weaker than (\ref{eq:M_n_unifom_positive}) and perhaps
at the expense of strengthening (\ref{eq:g_positive_somewhere}).

\subsection{The case of finitely many observations}

In practice, typically only a finite number of observations are available,
say $y_{0},\ldots,y_{m}$ for some $m\in\mathbb{T}$, and one aims
to sample from conditional distributions of the form $\mathbf{P}(X_{n}\in\cdot|\mathcal{F}_{[m,0]}^{Y})$,
for some $n\geq m$. There are a number of ways the coupling method
can be applied in this situation.

As an example, fix $m\in\mathbb{T}$ and suppose for simplicity of
exposition that $M_{n}$ does not depend on $n$, and has unique invariant
distribution $\pi$. Let the probability space of Section \ref{sub:The-coupling-for}
be augmented so as to also support an $E$-valued random variable
$Z_{m}$ such that with $y$ fixed, $\mathbf{Q}^{y}$ makes $Z_{m}$
independent of $(\xi_{n}^{x};x\in E;n\in\mathbb{T})$, and
\[
\mathbf{Q}^{y}(Z_{m}=x)=\frac{\pi(x)g_{m}(x,y_{m})\phi_{m+1}(x,y_{m+1:0})}{\sum_{x^{\prime}\in E}\pi(x^{\prime})g_{m}(x^{\prime},y_{m})\phi_{m+1}(x^{\prime},y_{m+1:0})}=:\bar{\pi}_{m}^{y}(x).
\]
Also introduce a random variable $Z_{0}$ such that on the event $\{T_{0}\geq m\}$,
$Z_{0}:=\Phi_{T_{0},0}(x)$ for an arbitrary $x\in E$; and on the
event $\{T_{0}<m\}$, $Z_{0}:=\Phi_{m,0}(Z_{m})$. Then using the
fact that on the event $\{T_{0}\geq m\}$, $\Phi_{T_{0},0}(x)=\Phi_{m,0}(x^{\prime})$
for all $x^{\prime}$, we have
\begin{eqnarray*}
 &  & \mathbf{Q}^{y}(Z_{0}=z)\\
 &  & =\mathbf{Q}^{y}(\{Z_{0}=z\}\cap\{T_{0}\geq m\})+\mathbf{Q}^{y}(\{Z_{0}=z\}\cap\{T_{0}<m\})\\
 &  & =\mathbf{Q}^{y}(\{\Phi_{T_{0},0}(x)=z\}\cap\{T_{0}\geq m\})+\mathbf{Q}^{y}(\{\Phi_{m,0}(Z_{m})=z\}\cap\{T_{0}<m\})\\
 &  & =\mathbf{Q}^{y}(\{\Phi_{m,0}(Z_{m})=z\}\cap\{T_{0}\geq m\})+\mathbf{Q}^{y}(\{\Phi_{m,0}(Z_{m})=z\}\cap\{T_{0}<m\})\\
 &  & =\mathbf{Q}^{y}(\Phi_{m,0}(Z_{m})=z)=\sum_{x^{\prime}\in E}\bar{\pi}_{m}^{y}(x^{\prime})M_{m,0}^{y}(x^{\prime},z),
\end{eqnarray*}
and it is easily checked that $\sum_{x^{\prime}\in E}\bar{\pi}_{m}^{Y}(x^{\prime})M_{m,0}^{Y}(x^{\prime},z)=\mathbf{P}(X_{0}=z|\mathcal{F}_{[m,0]}^{Y}),$
$\mathbf{P}$-a.s.

Some modifications of Algorithm~\ref{alg:Perfect-sampling-for-hmm}
facilitate the sampling of $Z_{0}$. The ``while'' line is replaced
by:
\[
\text{while }card(\text{image of }\Phi_{n,0})>1\text{ and }n>m
\]
and the ``return'' line is replaced by
\begin{eqnarray*}
 &  & \text{if }card(\text{image of }\Phi_{n,0})=1\\
 &  & \quad\text{return }Z_{0}=\Phi_{n,0}(x),\text{ for any }x\in E,\\
 &  & \text{else}\\
 &  & \quad\text{sample }Z_{m}\text{ from the distribution on \ensuremath{E}with }\\
 &  & \quad prob(Z_{m}=x)\propto\pi(x)g_{m}(x,y_{m})\phi_{m+1}(x,y_{m+1:0})\\
 &  & \text{\quad and return \ensuremath{Z_{0}=\Phi_{m,0}(Z_{m})}.}
\end{eqnarray*}
The resulting procedure may be computationally cheaper than direct
calculation and sampling from $\sum_{x\in E}\bar{\pi}_{m}^{y}(x)M_{m,0}^{y}(x,\cdot)$
if $T_{0}\gg m$.

\subsection{Numerical Examples}

\subsubsection{Sensitivity to model misspecification\label{sub:Sensitivity-to-model}}

The purpose of this example is to numerically investigate an HMM for
which the coupling is almost surely successful, i.e. $\lim_{n\rightarrow-\infty}\beta(M_{n,k}^{Y})=0$,
$\mathbf{P}$-a.s., but for which $\beta(M_{n,k}^{Y})$ converges
to zero very slowly or perhaps even remains bounded away from zero
as $n\to-\infty$ when the HMM is misspecified, in the sense that
$Y$ is not distributed according to $\mathbf{P}$.

Consider $E=\{1,2,3\}$, $F=\mathbb{R}$, $\psi$ Lebesgue measure,
and a time-homogeneous HMM, i.e. Assumption~\ref{assu:time_homog}
holds, with for some $\delta\in(0,1)$, $M_{0}$ written in matrix
form:
\[
M_{0}=\begin{bmatrix}1-\delta & \delta & 0\\
\delta/2 & 1-\delta & \delta/2\\
0 & \delta & 1-\delta
\end{bmatrix},
\]
\[
g_{0}(1,y)=g_{0}(3,y)=\frac{e^{-y^{2}/2}}{\sqrt{2\pi}},\quad g_{0}(2,y)=\frac{e^{-(y-1)^{2}/2}}{\sqrt{2\pi}}.
\]
Clearly Assumption~\ref{assu:non_degeneracy} holds and (\ref{eq:M_0_m_step})
holds with $m=2$. Hence by the arguments of Section~\ref{sub:Almost-sure-successful},
for all $k\in\mathbb{T}$, $\lim_{n\rightarrow-\infty}\beta(M_{n,k}^{Y})=0$,
$\mathbf{P}$-a.s.

Figure~\ref{fig:misspec} illustrates $\beta(M_{n,0}^{y})$ and histograms
of the coupling time $T_{0}$ obtained from $10^{4}$ independent
runs of Algorithm~\ref{alg:Perfect-sampling-for-hmm}, for three
different data sequences $y$. The first, corresponding to the left
column of plots, was drawn from $\mathbf{P}$, with the true sequence
of hidden states also shown. The second, corresponding to the middle
column, is a sample path of the process: $Y_{0}=0$ and $Y_{n}=Y_{n+1}+V_{n}$,
where the $V_{n}$ are i.i.d. $\mbox{\ensuremath{\mathcal{N}}}(0,0.25)$.
The third, corresponding to right column, is a realization of $Y_{n}=0.003n+V_{n}$.
To interpret these plots, note that for $x\in\{1,3\}$, $\lim_{y_{0}\to-\infty}g(2,y_{0})/g(x,y_{0})=0$,
and similarly, if it were true that the observation sequence were
constant $y_{n}=y_{n+1}=\cdots=y_{0}$, elementary manipulations show
that for $x\in\{1,3\}$, $\lim_{y_{0}\to-\infty}M_{n,0}^{y}(x,x)=1$,
hence $\lim_{y_{0}\to-\infty}\beta(M_{n,0}^{y})=1$. The plots in
the second and third columns reflect a similar phenomenon, namely
that long sequences of negative observations may slow down the convergence
to zero of $\beta(M_{n,0}^{y})$ as $n\to-\infty$. In the case of
the third column, it is notable that $\beta(M_{n,0}^{y})$ appears
to be bounded away from zero, indeed the same phenomenon was observed
in much longer runs of the algorithm (numerical results not shown).

\begin{figure}
\hfill{}\includegraphics[bb=110bp 280bp 489bp 565bp,clip,width=1\textwidth]{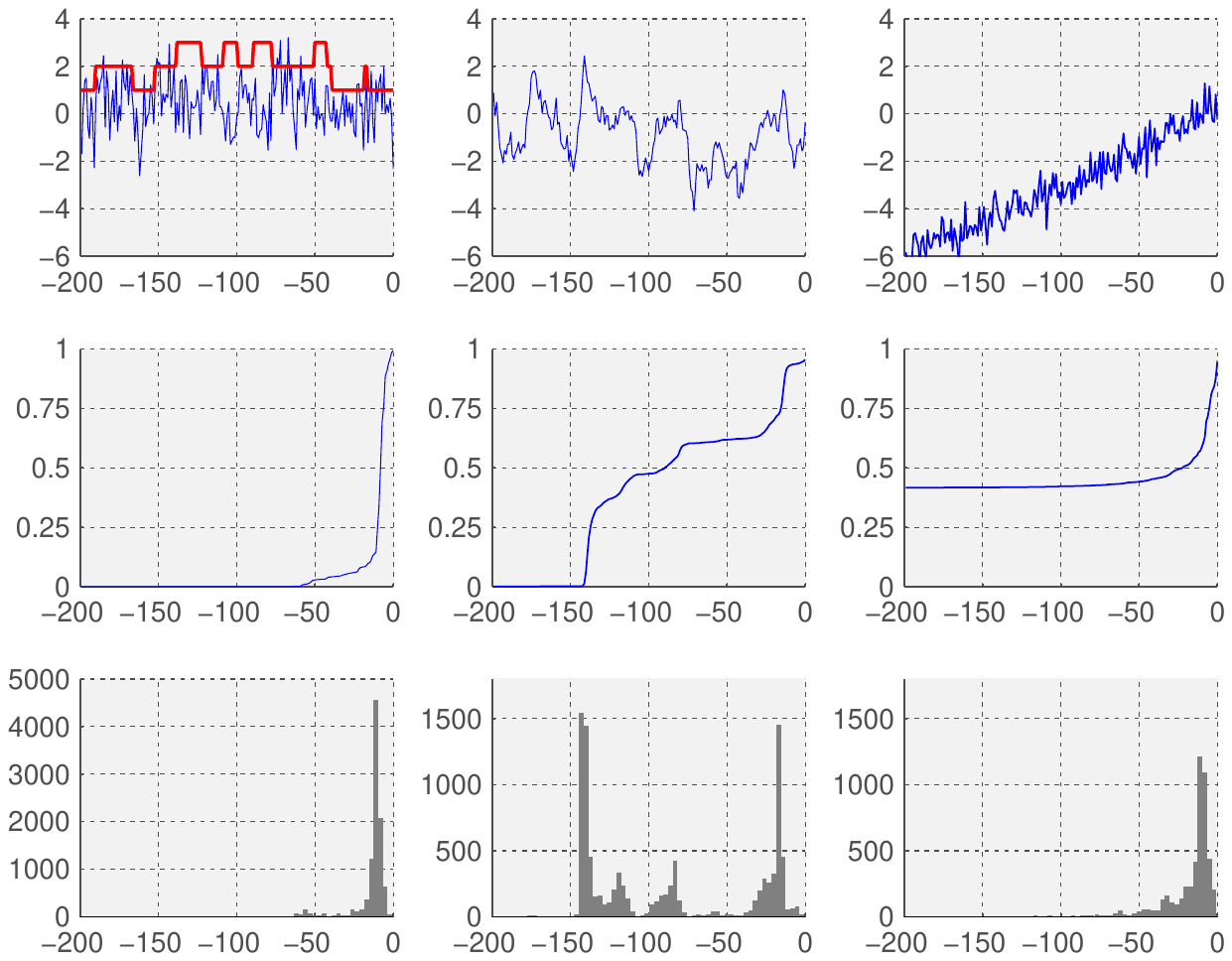}\hfill{}\protect\caption{\label{fig:misspec}Top row: $y_{n}$ vs. $n$ (top left plot also
shows true sequence of hidden states). Middle row: $\beta(M_{n,0}^{y})$
vs. $n$. Bottom row: histograms of $T_{0}$ obtained from $10^{4}$
runs of the algorithm. First column corresponds to data simulated
from the HMM, second and third columns correspond to data from misspecified
models. In the first and second cases, all of the $10^{4}$ realizations
of $T_{0}$ were valued within $\{-200,\ldots,0\}$, for the third
case, only $50.3\%$ of the $10^{4}$ realizations of $T_{0}$ were
valued in $\{-200,\ldots,0\}$, the remaining realizations are not
shown on the bottom-right histogram.}
\end{figure}

\subsubsection{Simulating multiple samples}

Running Algorithm~\ref{alg:Perfect-sampling-for-hmm} several times
in order to obtain multiple i.i.d. samples from $\pi_{0}^{y}$ may
be prohibitively expensive. Consider the following procedure.
\begin{enumerate}
\item Fix $n$, compute $M_{n,0}^{y}$ and obtain an exact sample $X_{n}^{\star}$
from $\pi_{n}^{y}$ using the perfect sampling scheme.
\item Given $X_{n}^{\star}$, drawn $N$ conditionally independent samples,
\[
X_{0}^{(i)}\sim M_{n,0}^{y}(X_{n}^{\star},\cdot),\quad i=1,\ldots,N.
\]

\end{enumerate}
Since $X_{n}^{\star}\sim\pi_{n}^{y}$ and $\sum_{x}\pi_{n}^{y}(x)M_{n,0}^{y}(x,\cdot)=\pi_{0}^{y}(\cdot)$,
the samples $(X_{0}^{(i)};i=1,\ldots,N)$ each have marginal distribution
$\pi_{0}^{y}$, but are not independent in general. Indeed writing
$\lambda_{n}^{y}$ for the joint distribution of $(X_{0}^{(1)},X_{0}^{(2)})$,
\begin{eqnarray*}
 &  & \left\Vert \lambda_{n}^{y}-\pi_{0}^{y}\otimes\pi_{0}^{y}\right\Vert \\
 &  & \quad\quad=\frac{1}{2}\sum_{x,x^{\prime}}\left|\sum_{z}\pi_{n}^{y}(z)M_{n,0}^{y}(z,x)M_{n,0}^{y}(z,x^{\prime})-\pi_{0}^{y}(x)\pi_{0}^{y}(x^{\prime})\right|.
\end{eqnarray*}
Conditional ergodicity dictates this quantity converges to zero as
$n\to-\infty$. Table~\ref{tab:_multiple samples} shows numerical
values against $n$, for the HMM of Section~\ref{sub:Sensitivity-to-model}
with the data sequence shown in the top-left plot of Figure~\ref{fig:misspec}
and with $\pi_{n}^{y}$ and $\pi_{0}^{y}$ approximated by $M_{-1000,n}^{y}(1,\cdot)$
and $M_{-1000,0}^{y}(1,\cdot)$. Also shown is the average percentage
of CPU time spent on step 1. of the above two step procedure, obtained
from $10^{4}$ independent runs. The algorithms were implemented in
Matlab on a 2.80GHz desktop PC. These results illustrate that step
1., which amounts to guaranteeing that the marginal distribution of
each $X_{0}^{(i)}$ is exactly $\pi_{0}^{y}$, is relatively cheap
when $N$ is large, even when $n$ is large enough to make the pairwise
dependence between $X_{0}^{(1)}$ and $X_{0}^{(2)}$ negligible.

\begin{table}[H]
\begin{center}
\begin{tabular}{cc|ccccc}
\toprule
& &    $n=5$ & $n=10$ & $n=25$ & $n=50$  & $n=100$  \\ \midrule
\multicolumn{2}{c|}{$\|\lambda^y_n - \pi^y_0\otimes\pi^y_0\|$ }  & 0.27  & 0.014 & 0.0033 & $3.3 \times10^{-4}$ & $<10^{-6}$\\ \midrule
\% time & $N=10^2$ & 23 (25)    & 26 (29)    & 41 (27)   & 43 (10)   & 47 (22)\\
 step 1.      & $N=10^3$ & 4.2 (4.5)  & 4.8 (5.6)  & 9.2 (5.1) & 9.8 (1.6) & 11 (3.8) \\
       & $N=10^4$ & 0.53 (0.57)& 0.61 (0.72)& 1.2 (0.66)& 1.3 (0.20)& 1.5 (0.49) \\

\bottomrule
\end{tabular}
\end{center}

\protect\caption{\label{tab:_multiple samples}In the second row, $\lambda_{n}^{y}:=\mathrm{Law}(X_{0}^{(1)},X_{0}^{(2)})$,
where $X_{0}^{(1)},X_{0}^{(2)}$ are obtained from step 2. of the
procedure, i.e. $X_{0}^{(i)}\sim M_{n}^{y}(X_{n}^{\star},\cdot)$.
The third to fifth rows show mean percentage of overall CPU time spent
on step 1., with estimated standard deviation in parentheses, obtained
from $10^{4}$ runs.}
\end{table}

\subsection{Outlook}

How much of all this can be generalized beyond the case in which $E$
is a finite set? We believe: quite a lot, although the work involved
is non-trivial. Of course if $E$ is not a finite set, then we have
to let go of Fact~\ref{fact:there-always-exists}; without further
assumption there is no guarantee that even a single sequence of absolute
probabilities for $M$ exists. The coupling we have specified in Section~\ref{sec:Perfect-sampling}
relies heavily on the fact that $E$ contains only finitely many points,
but a generalization via the kind of mechanisms used for backward
coupling of homogeneous chains, see for example \citep{foss1998perfect}
and references therein, may be feasible, and that would be the starting
point from which to investigate generalization of Theorem~\ref{thm:perfect sampling-1}.
Quite a few of the arguments used in the proof of Theorem~\ref{thm:tail}
do not really rely on $E$ being a finite set. Regarding the application
to HMM's, as soon as $E$ contains infinitely many points, then with
a few exceptions such as the linear-Gaussian state-space model, the
functions $\phi_{n}(x,y_{n:0})$ are not available in closed form,
so sampling from the kernels $M_{n}^{y}$ becomes non-trivial and
the perfect simulation algorithm may lose its practical relevance.
Overall though, there are several possible avenues for further investigation.

\section{Appendix}
\begin{proof}
\emph{(of Lemma~\ref{lem:M^Y_is_a_version}.)}First, we claim that
for any nonnegative measurable function $f$, the following holds
$\mathbf{P}$-a.s.,
\begin{eqnarray}
 &  & \mathbf{E}[f(X_{n-1},Y_{n},\ldots,Y_{0})|\sigma(X_{n-1})]\label{eq:phi_equals_cond_exp}\\
 &  & =\int f(X_{n-1},y_{n},\ldots,y_{0})\phi_{n}(X_{n-1},y_{n:0})\psi^{\otimes(|n|+1)}(d(y_{n},\ldots,y_{0})).\nonumber
\end{eqnarray}
The r.h.s. of (\ref{eq:phi_equals_cond_exp}) is clearly measurable
w.r.t. $\sigma(X_{n-1})$, so to prove the claim it remains to check
that for any $x_{n-1}\in E$,
\begin{eqnarray}
 &  & \int_{A(x_{n-1})}\int f(X_{n-1},y_{n},\ldots,y_{0})\phi_{n}(X_{n-1},y_{n:0})\psi^{\otimes(|n|+1)}(d(y_{n},\ldots,y_{0}))d\mathbf{P}\nonumber \\
 &  & \quad\quad=\int_{A(x_{n-1})}f(X_{n-1},Y_{n},\ldots,Y_{0})d\mathbf{P},\label{eq:phi_equals_cond_exp_1}
\end{eqnarray}
where $A(x_{n-1})$ is the event $\{X_{n-1}=x_{n-1}\}$. It follows
from (\ref{eq:P_HMM_fidi}) and by writing out the definition of $\phi_{n}$
in (\ref{eq:phi_defn}) that the l.h.s. of (\ref{eq:phi_equals_cond_exp_1})
is equal to
\begin{eqnarray*}
 &  & \pi_{n-1}(x_{n-1})\int f(x_{n-1},y_{n},\ldots,y_{0})\phi_{n}(x_{n-1},y_{n:0})\psi^{\otimes(|n|+1)}(d(y_{n},\ldots,y_{0}))\\
 &  & =\int\pi_{n-1}(x_{n-1})f(x_{n-1},y_{n},\ldots,y_{0})\sum_{(x_{n},\ldots,x_{0})}\prod_{k=n}^{0}M_{k}(x_{k-1},x_{k})G_{k}(x_{k},dy_{k}),
\end{eqnarray*}
which is also equal to the r.h.s. (\ref{eq:phi_equals_cond_exp_1}),
thus completing the proof of (\ref{eq:phi_equals_cond_exp}).

Next note that $M_{n}^{Y}(X_{n-1},x)$ is measurable w.r.t. $\mathcal{F}_{[n,0]}^{Y}\vee\sigma(X_{n-1})$,
so in order to complete the proof of the Lemma it remains, by a standard
monotone class argument, to show:
\begin{eqnarray}
 &  & \int_{\{X_{n-1}=x_{n-1}\}}\mathbb{I}[(Y_{n},\ldots,Y_{0})\in A]M_{n}^{Y}(X_{n-1},x)d\mathbf{P}\nonumber \\
 &  & =\mathbf{P}(\{X_{n-1}=x_{n-1}\}\cap\{(Y_{n},\ldots,Y_{0})\in A\}\cap\{X_{n}=x\}),\label{eq:M_y_obj_cond_exp}
\end{eqnarray}
for any $x_{n-1},x\in E$ and $A\in\mathcal{B}(F)^{\otimes(|n|+1)}$.
We proceed by fixing $x$ and applying (\ref{eq:phi_equals_cond_exp})
with $f(x_{n-1},y_{n},\ldots,y_{0})=\mathbb{I}[(y_{n},\ldots,y_{0})\in A]M_{n}^{y}(x_{n-1},x)$,
we have using the definitions of $M_{n}^{y}(x_{n-1},x)$, $\phi_{n+1}$
and $\mathbf{P}$ that the following equalities hold $\mathbf{P}$-a.s.,
\begin{eqnarray*}
 &  & \mathbf{E}[\mathbb{I}[(Y_{n},\ldots,Y_{0})\in A]M_{n}^{Y}(X_{n-1},x)|\sigma(X_{n-1})]\\
 &  & =\int_{\{(y_{n},\ldots,y_{0})\in A\}}M_{n}(X_{n-1},x)g_{n}(x,y_{n})\phi_{n+1}(x,y_{n+1:0})\psi^{\otimes(|n|+1)}(d(y_{n},\ldots,y_{0}))\\
 &  & =\mathbf{P}(\{(Y_{n},\ldots,Y_{0})\in A\}\cap\{X_{n}=x\}|\sigma(X_{n-1})),
\end{eqnarray*}
using this identity and the tower property of conditional expectation,
we can re-write the l.h.s.\ of (\ref{eq:M_y_obj_cond_exp}) as
\begin{eqnarray*}
 &  & \int_{\{X_{n-1}=x_{n-1}\}}\mathbb{I}[(Y_{n},\ldots Y_{0})\in A]M_{n}^{Y}(X_{n-1},x)d\mathbf{P}\\
 &  & =\int_{\{X_{n-1}=x_{n-1}\}}\mathbf{E}[\mathbb{I}[(Y_{n},\ldots,Y_{0})\in A]M_{n}^{Y}(X_{n-1},x)|\sigma(X_{n-1})]d\mathbf{P}\\
 &  & =\int_{\{X_{n-1}=x_{n-1}\}}\mathbf{P}(\{(Y_{n},\ldots,Y_{0})\in A\}\cap\{X_{n}=x\}|\sigma(X_{n-1}))d\mathbf{P}\\
 &  & =\mathbf{P}(\{X_{n-1}=x_{n-1}\}\cap\{(Y_{n},\ldots,Y_{0})\in A\}\cap\{X_{n}=x\}).
\end{eqnarray*}
The equality (\ref{eq:M_y_obj_cond_exp}) therefore holds and this
completes the proof of the lemma.\end{proof}
\begin{rem}
We note that the arguments of the above proof rely on the definition
in (\ref{eq:M^y_defn}) only through the values taken by $M_{n}^{y}(x,\cdot)$
on the support of $\phi_{n}$.\end{rem}
\begin{proof}
\emph{(of Lemma~\ref{lem:F_X_F_X_tilde})}. To prove $\mathcal{F}_{I}^{X}=\{A\times\Omega^{Y};A\in\mathcal{F}_{I}^{\tilde{X}}\}$
we need to show that $\mathcal{C}_{I}^{X}:=\{A\times\Omega^{Y};A\in\mathcal{F}_{I}^{\tilde{X}}\}$
is the smallest $\sigma$-algebra of subsets of $\Omega$ w.r.t.\ which
all the $(X_{n})_{n\in I}$ are measurable. We break this down into
three steps: i) show that $\mathcal{C}_{I}^{X}$ is a $\sigma$-algebra;
ii) show that every $(X_{n})_{n\in I}$ is measurable w.r.t. $\mathcal{C}_{I}^{X}$;
iii) show that if any set is removed from $\mathcal{C}_{I}^{X}$ then
the resulting collection of sets either doesn't contain $X_{n}^{-1}(A)$
for some $n\in I$ and $A\in\mathcal{B}(E)$, or is not a $\sigma$-algebra.

Step i) is immediate since $\mathcal{F}_{I}^{\tilde{X}}$ is by definition
a $\sigma$-algebra and $\Omega^{Y}$ is non-empty (because $F$ is
by definition non-empty). For step ii), we have by definition of $\mathcal{F}_{I}^{\tilde{X}}$
that for any $n\in I$ and $A\in\mathcal{B}(E)$, $\tilde{X}_{n}^{-1}(A)\in\mathcal{F}_{I}^{\tilde{X}}$,
and $X_{n}^{-1}(A)=\eta^{-1}\circ\tilde{X}_{n}^{-1}(A)=\tilde{X}_{n}^{-1}(A)\times\Omega^{Y}$,
hence $X_{n}^{-1}(A)\in\mathcal{C}_{I}^{X}$. For step iii), for an
arbitrary $B\in\mathcal{F}_{I}^{\tilde{X}}$ let us remove the set
$B\times\Omega^{Y}$ from $\mathcal{C}_{I}^{X}$, the resulting collection
of sets being $\{A\times\Omega^{Y};A\in\mathcal{F}_{I}^{\tilde{X}}\setminus B\}=:\mathcal{D}_{I}^{X}$.
Since $\mathcal{F}_{I}^{\tilde{X}}$ is the smallest $\sigma$-algebra
w.r.t.\ which all the $(\tilde{X}_{n})_{n\in I}$ are measurable,
either there exists some $n\in I$ and $A\in\mathcal{B}(E)$ such
that $\tilde{X}_{n}^{-1}(A)\notin\mathcal{F}_{I}^{\tilde{X}}\setminus B$,
or $\mathcal{F}_{I}^{\tilde{X}}\setminus B$ is not a $\sigma$-algebra.
In the former case, $X_{n}^{-1}(A)=\eta^{-1}\circ\tilde{X}_{n}^{-1}(A)=\tilde{X}_{n}^{-1}(A)\times\Omega^{Y}\notin\mathcal{D}_{I}^{X}$,
i.e. $X_{n}$ is not measurable w.r.t. $\mathcal{D}_{I}^{X}$. In
the latter case, we claim that $\mathcal{D}_{I}^{X}$ is not a $\sigma$-algebra.
To prove this claim we shall argue to the contrapositive that for
$\tilde{\mathcal{C}}$ any collection of subsets of $E^{\mathbb{T}}$,
if $\mathcal{D}:=\{A\times\Omega^{Y};A\in\mathcal{\tilde{C}}\}$ is
a $\sigma$-algebra, then $\tilde{\mathcal{C}}$ is a $\sigma$-algebra.
To this end, observe: if $\mathcal{D}$ contains $\Omega$, then $\tilde{\mathcal{C}}$
contains $\Omega^{X}$; if $\mathcal{D}$ is closed under complements,
then $A\in\tilde{\mathcal{C}}\Rightarrow A\times\Omega^{Y}\in\mathcal{D}\Rightarrow(A\times\Omega^{Y})^{c}\in\mathcal{D}\Rightarrow A^{c}\times\Omega^{Y}\in\mathcal{D}\Rightarrow A^{c}\in\tilde{\mathcal{C}}$,
i.e. $\tilde{\mathcal{C}}$ is closed under complements; if $\mathcal{D}$
is closed under countable unions, $A_{n}\in\tilde{\mathcal{C}}\Rightarrow A_{n}\times\Omega^{Y}\in\mathcal{D}\Rightarrow\cup_{n\in\mathbb{N}}(A_{n}\times\Omega^{Y})\in\mathcal{D}\Rightarrow(\cup_{n\in\mathbb{N}}A_{n})\times\Omega^{Y}\in\mathcal{D}\Rightarrow\cup_{n}A_{n}\in\tilde{\mathcal{C}}$
, i.e. $\tilde{\mathcal{C}}$ is closed under countable unions. This
completes the proof of $\mathcal{F}_{I}^{X}=\{A\times\Omega^{Y};A\in\mathcal{F}_{I}^{\tilde{X}}\}$,
from which (\ref{eq:tail_x_tail_x_tilde}) follows directly.
\end{proof}

\begin{proof}
\emph{(of Lemma~\ref{lem:P_bold_P})} As a consequence of Lemma~\ref{lem:existence_of_pi^y},
there exists $H\in\mathcal{F}^{Y}$ with $\mathbf{P}(H)=1$ such that
for all $\omega\in H$, $\sum_{z\in E}\mu_{n-1}^{Y(\omega)}(z)M_{n}^{Y(\omega)}(z,x)=\mu_{n}^{Y(\omega)}(x)$
for all $n$ and $x$. Set $\tilde{H}=Y(H)$. For $y\in\tilde{H}$
we are assured by the usual extension argument of the existence of
$P^{y}(\cdot)$ a measure with the desired properties. For $y\notin\tilde{H}$
set $P^{y}(\cdot)$ to an arbitrary probability. We thus obtain the
desired kernel. It follows from Lemma~\ref{lem:F_X_F_X_tilde} that
every set in $\mathcal{F}^{X}$ is of the form $A\times\Omega^{Y}$
for some $A\in\mathcal{F}^{\tilde{X}}$, and then $\mathbf{P}^{\mathcal{F}^{Y}}$
is a probability kernel because $P$ is. In order to establish (\ref{eq:P_f_is_a_version}),
we argue as follows. With $\mathbf{P}\circ Y^{-1}$ the push-forward
of $\mathbf{P}$ by $Y$, define $\tilde{\mathbf{P}}(A):=\int_{\Omega^{X}\times\Omega^{Y}}\mathbb{I}[(x,y)\in A]P^{y}(dx)(\mathbf{P}\circ Y)(dy)$,
which is a probability measure on $(\Omega,\mathcal{F})$ and by construction
$\tilde{\mathbf{P}}(A\times\Omega^{Y}|\mathcal{F}^{Y})(\omega)=\mathbf{P}^{\mathcal{F}^{Y}}(\omega,A\times F)$,
$\tilde{\mathbf{P}}$-a.s., for each $A\in\mathcal{F}^{\tilde{X}}$.
The proof of (\ref{eq:P_f_is_a_version}) will be complete if we can
show that $\tilde{\mathbf{P}}=\mathbf{P}$, since then $\tilde{\mathbf{P}}(\cdot|\mathcal{F}^{Y})=\mathbf{P}(\cdot|\mathcal{F}^{Y})$.
For $\tilde{\mathbf{P}}=\mathbf{P}$ it is sufficient that for each
$n\in\mathbb{T}$ and $A\in\mathcal{F}_{n}$, $\tilde{\mathbf{P}}(A)=\mathbf{P}(A)$,
and the latter holds since, using (\ref{eq:hmm_cond_ind_1}), (\ref{eq:hmm_cond_ind_2}),
Lemmata~\ref{lem:M^Y_is_a_version} and~\ref{lem:existence_of_pi^y},
\begin{eqnarray*}
 &  & \tilde{\mathbf{P}}(A)\\
 &  & =\int\sum_{(x_{n},\ldots,x_{0})\in E^{n+1}}\mathbb{I}[(x,y)\in A]\mu_{n}^{y}(x_{n})\prod_{k=n+1}^{0}M_{k}^{y}(x_{k-1},x_{k})(\mathbf{P}\circ Y)(dy)\\
 &  & =\mathbf{E}[\mathbf{E}[\mathbf{E}[\cdots\mathbf{E}[\mathbb{I}[(X,Y)\in A]|\mathcal{F}^{Y}\vee\mathcal{F}_{]-\infty,0]}^{X}]\cdots|\mathcal{F}^{Y}\vee\mathcal{F}_{]-\infty,n]}^{X}]|\mathcal{F}^{Y}]]\\
 &  & =\mathbf{P}(A).
\end{eqnarray*}
The proof of the lemma is complete.
\end{proof}

\begin{proof}
\emph{(of Lemma~\ref{lem:beta_My_bounds})}. Throughout the proof
fix $y\in\Omega^{Y}$. For 1., first note that by (\ref{eq:phi_defn})
and (\ref{eq:M_n,k_mixing}),
\begin{equation}
\epsilon^{-}\nu(g\phi)_{n+2}^{y}\leq\phi_{n+1}(x,y_{n+1:0})\leq\epsilon^{+}\nu(g\phi)_{n+2}^{y},\label{eq:phi_upp_low}
\end{equation}
where
\[
\nu(g\phi)_{n+2}^{y}:=\sum_{z}\nu(z)g_{n+1}(z,y_{n+1})\phi_{n+2}(z,y_{n+2:0})>0,
\]
the positivity being due to the hypotheses of 1. combined with (\ref{eq:phi_upp_low})
and $\epsilon^{+}>0$. It follows from (\ref{eq:recurse My}), (\ref{eq:M^y_defn}),
the hypothesis of 1., (\ref{eq:M_n,k_mixing}) and (\ref{eq:phi_upp_low})
that
\begin{eqnarray*}
 & M_{n,k}^{y}(x,x^{\prime}) & =M_{n+1}^{y}(x,x^{\prime})\\
 &  & =\dfrac{M_{n+1}(x,x^{\prime})g_{n+1}(x^{\prime},y_{n+1})\phi_{n+2}(x^{\prime},y_{n+2:0})}{\phi_{n+1}(x,y_{n+1:0})}\\
 &  & \geq\frac{\epsilon^{-}}{\epsilon_{+}}\frac{\nu(x^{\prime})g_{n+1}(x^{\prime},y_{n+1})\phi_{n+2}(x^{\prime},y_{n+2:0})}{\nu(g\phi)_{n+2}^{y}},
\end{eqnarray*}
thus there exists a probability distribution $\tilde{\nu}_{n+1}^{y}$
such that $M_{n+1}^{y}(x,x^{\prime})\geq\frac{\epsilon^{-}}{\epsilon_{+}}\tilde{\nu}_{n+1}^{y}(x^{\prime})$.
Combining this fact with the expression for $\beta(\cdot)$ in (\ref{eq:dobrushin_alt})
gives (\ref{eq:M^y_dobrushin_bound1}).

A simple induction shows that when (\ref{eq:phi_k_positive}) and
the hypotheses of 2. hold, $\phi_{j}(x,y_{j:0})>0$ for all $x$ and
$j=k+1,k,\ldots,n+1$. Combining this fact with (\ref{eq:recurse My}),
(\ref{eq:M^y_defn}), (\ref{eq:phi_defn}) and (\ref{eq:M_n,k_mixing}),
\begin{eqnarray*}
 &  & M_{n,k}^{y}(x_{n},x_{k})\\
 &  & =\frac{\sum_{(x_{n+1},\ldots,x_{k-1})}\left(\prod_{j=n+1}^{k}M_{j}(x_{j-1},x_{j})g_{j}(x_{j},y_{j})\right)\phi_{k+1}(x_{k},y_{k+1:0})}{\phi_{n+1}(x_{n},y_{n+1:0})}\\
 &  & \geq\frac{M_{n,k}(x_{n},x_{k})\phi_{k+1}(x_{k},y_{k+1:0})}{\sum_{z}M_{n,k}(x_{n},z)\phi_{k+1}(z,y_{k+1:0})}\prod_{j=n+1}^{k}\frac{g_{j}^{-}(y_{j})}{g_{j}^{+}(y_{j})}\\
 &  & \geq\frac{\epsilon^{-}}{\epsilon^{+}}\frac{\nu(x_{k})\phi_{k+1}(x_{k},y_{k+1:0})}{\sum_{z}\nu(z)\phi_{k+1}(z,y_{k+1:0})}\prod_{j=n+1}^{k}\frac{g_{j}^{-}(y_{j})}{g_{j}^{+}(y_{j})},
\end{eqnarray*}
and the final denominator is strictly positive due to (\ref{eq:phi_k_positive}).
Using again (\ref{eq:dobrushin_alt}) gives (\ref{eq:M^y_dobrushin_bound_2}).
\end{proof}
\bibliographystyle{plainnat}
\bibliography{HMM}

\end{document}